\documentclass[12pt]{amsart}

\newcommand{\indentalign}{\hspace{0.3in}&\hspace{-0.3in}}
\newcommand{\la}{\langle}
\newcommand{\ra}{\rangle}

\newcommand{\sech}{\operatorname{sech}}
\newcommand{\defeq}{\stackrel{\rm{def}}{=}}
\newcommand{\spn}{\operatorname{span}}

\usepackage{amssymb}
\usepackage{amsthm}
\usepackage{amsxtra}
\usepackage{graphicx}
\usepackage{color}

\newtheorem{theorem}{Theorem}
\newtheorem{definition}[theorem]{Definition}
\newtheorem{proposition}{Proposition}[section]
\newtheorem{lemma}[proposition]{Lemma}
\newtheorem{corollary}[proposition]{Corollary}

\theoremstyle{remark}
\newtheorem{remark}[proposition]{Remark}

\numberwithin{equation}{section}

\ifx\pdfoutput\undefined
  \DeclareGraphicsExtensions{.pstex, .eps}
\else
  \ifx\pdfoutput\relax
    \DeclareGraphicsExtensions{.pstex, .eps}
  \else
    \ifnum\pdfoutput>0
      \DeclareGraphicsExtensions{.pdf}
    \else
      \DeclareGraphicsExtensions{.pstex, .eps}
    \fi
  \fi
\fi

\title[Dynamics of KdV solitons]
{Dynamics of KdV solitons in the presence of a slowly varying potential}
\author{Justin Holmer}
\address{Brown University, Department of Mathematics, Box 1917, Providence, RI 02912, USA}
\email{holmer@math.brown.edu}

\begin{document}

\maketitle

\begin{abstract}
We study the dynamics of solitons as solutions to the perturbed KdV (pKdV) equation $\partial_t u = -\partial_x ( \partial_x^2 u + 3u^2-bu)$, where $b(x,t) = b_0(hx,ht)$, $h\ll 1$ is a slowly varying, but not small, potential.  We obtain an explicit description of the trajectory of the soliton parameters of scale and position on the dynamically relevant time scale $\delta h^{-1}\log h^{-1}$, together with an estimate on the error of size $h^{1/2}$.  
 In addition to the Lyapunov analysis commonly applied to these problems, we use a local virial estimate due to Martel-Merle \cite{MM}.  The results are supported by numerics.  The proof does not rely on the inverse scattering machinery and is expected to carry through for the $L^2$ subcritical gKdV-$p$ equation, $1<p<5$.  The case of $p=3$, the modified Korteweg-de Vries (mKdV) equation, is structurally simpler and more precise results can be obtained by the method of Holmer-Zworski \cite{HZ2}. 
\end{abstract}

\section{Introduction}
\label{S:intro}

The Korteweg-de Vries (KdV) equation
\begin{equation}
\label{E:KdV}
\partial_t u = \partial_x ( -\partial_x^2 u - 3u^2)
\end{equation}
is globally well-posed in $H^k$ for $k\geq 1$ (see Kenig-Ponce-Vega \cite{KPV}).  It possesses soliton solutions $u(t,x) = \eta(x, a+4c^2t, c)$, where $\eta(x,a,c) = c^2\theta(c(x-a))$ and $\theta(y) = 2\sech^2y$ (so that $\theta''+3\theta^2=4\theta$).   Benjamin \cite{Benjamin}, Bona \cite{Bona},  and Bona-Souganidis-Strauss \cite{BSS} showed that these solitons are orbitally stable under perturbations of the initial data.  We consider here the behavior of these solitons under structural perturbations, i.e. Hamiltonian perturbations of the equation \eqref{E:KdV} itself.  Dejak-Sigal \cite{DS}, motivated by a model of shallow water wave propagation over a slowly-varying bottom, have considered the perturbed KdV (pKdV)
\begin{equation}
\label{E:pKdV}
\partial_t u = \partial_x ( -\partial_x^2 u - 3u^2 + bu)
\end{equation}
where $b(x,t) = h^{1+\delta}b_0(hx,ht)$ and $h\ll 1$.  They proved that the effects of this potential are small on the dynamically relevant time frame.  We consider instead $b(x,t) = b_0(hx,ht)$, a slowly-varying \emph{but not small} potential,\footnote{Dejak-Sigal \cite{DS} state a more general result that appears to allow for potentials that are not small.  However, the smallness in their result is required to reach the dynamically relevant time frame $\sim h^{-1}$.  See the comments below in \S \ref{S:earlierwork}.} which allows for considerably richer dynamics.

 To state our main theorem, we need the following definition:
 
\begin{definition}[Asymptotic time-scale]
\label{D:time-scale}
Given $b_0\in C^\infty_c(\mathbb{R}^2)$, $A_0\in \mathbb{R}$, $C_0>0$, and $\delta>0$, let $A(\tau)$, $C(\tau)$ solve the system of ODEs
\begin{equation}
\label{E:ODEs-scaled}
\left\{
\begin{aligned}
&\dot A = 4C^2 - b_0(A,\cdot) \\
&\dot C = \frac13 C\partial_A b_0(A,\cdot)
\end{aligned}
\right.
\end{equation}
with initial data $A(0)=A_0$ and $C(0)=C_0$.
Let $T_*$ be the maximal time such that on $[0,T_*)$, we have $\delta\leq C(\tau)\leq \delta^{-1}$. ($T_*$ could be $+\infty$.)
\end{definition}

Let $\la u,v \ra = \int uv$. 

\begin{theorem}
\label{T:main}
Given $b_0\in C^\infty_c(\mathbb{R}^2)$, $A_0\in \mathbb{R}$, $C_0>0$, and $\delta>0$, let $T_*$ be the time defined in Def. \ref{D:time-scale}.  Let $a_0=h^{-1}A_0$ and $c_0=C_0$.  Then for $0\leq t\leq T \defeq h^{-1}\min(T_*, \delta \log h^{-1})$, there exist trajectories $a(t)$ and $c(t)$, and positive constants $\epsilon=\epsilon(\delta)$ and $C=C(\delta, b_0)$, such that the following holds.  Taking $u(t)$ the solution of \eqref{E:pKdV} with potential $b(x,t) = b_0(hx,ht)$ and initial data $\eta(\cdot, a_0,c_0)$, let $v(x,t) \defeq u(x,t) - \eta(x, a(t),c(t))$. Then 
\begin{equation}
\label{E:v-global}
\|v\|_{L_{[0,T]}^\infty H_x^1} \lesssim  h^{1/2}e^{Cht} \,,
\end{equation}
\begin{equation}
\label{E:v-local}
\|e^{-\epsilon |x-a|} v\|_{L_{[0,T]}^2H_x^1} \lesssim h^{1/2} e^{Cht} \,,
\end{equation}
and
\begin{equation}
\label{E:so}
\la v, \eta(\cdot,a,c) \ra =0 \,, \qquad \la v, (x-a) \eta(\cdot,a,c) \ra =0 \,.
\end{equation}
Moreover,
\begin{equation}
\label{E:traj-est}
|a(t) - h^{-1}A(ht)| \lesssim e^{Cht} \,, \qquad |c(t) - C(ht)| \lesssim he^{Cht} \,.
\end{equation}
\end{theorem}

Up to time $O(h^{-1})$, $a(t)$ is of size $O(h^{-1})$ and $c(t)$ is of size $O(1)$, and \eqref{E:traj-est} gives leading-order in $h$ estimates for $a(t)$ and $c(t)$ -- that is, despite the differences in magnitudes, the estimates for $a(t)$ and $c(t)$ provided by \eqref{E:traj-est} are equally strong.
The strength of the local estimate \eqref{E:v-local}, in comparison to the global estimate \eqref{E:v-global} on the error $v$, is that it involves integration in time over a (long) interval of length $O(h^{-1})$.  The estimate \eqref{E:v-local} is \textit{on par}, although slightly weaker than, the pointwise-in-time estimate $\|e^{-\epsilon |x-a|}v\|_{L_{[0,T]}^\infty L_x^2} \leq he^{Cht}$.   The two estimates \eqref{E:v-global}, \eqref{E:v-local} are consistent (but not equivalent to) $v$ being of amplitude $h$ but effectively supported over an interval of size $O(h^{-1})$, which is suggested by numerical simulations.  The trajectory estimates \eqref{E:traj-est} state that we can predict the center of the soliton to within accuracy $O(1)$ and the amplitude to within accuracy $O(h)$.  (This discussion does not include the $h^{-\delta}$ loss that occurs when passing to the natural Ehrenfest time scale $\delta h^{-1}\log h^{-1}$.)

To define the Hamiltonian structure associated with \eqref{E:pKdV}, let $J= \partial_x$ with
$$J^{-1}f(x) = \partial_x^{-1}f(x) \defeq \frac12 \left(\int_{-\infty}^x - \int_{x}^{+\infty}\right) f(y) \, dy \,.$$
We regard the function space $N=H^1(\mathbb{R})$ as a symplectic manifold with symplectic form $\omega(u,v) = \la u, J^{-1}v\ra$ densely defined on the tangent space $TN \simeq H^1$.  Then \eqref{E:pKdV} is the Hamilton flow $\partial_t u = JH'(u)$ associated with the Hamiltonian 
\begin{equation}
\label{E:Hamiltonian}
H = \frac12 \int (u_x^2 - 2u^3 +bu^2) \,.
\end{equation}
Let $M\subset N=H^1$ denote the two-dimensional submanifold of solitons
$$M = \{ \, \eta(\cdot, a,c) \; | \; a\in \mathbb{R} \,, \, c>0 \, \} \,.$$
By direct computation, we compute the restricted symplectic form $\omega\big|_M = 8c^2 da\wedge dc$ (thus $M$ is a \emph{symplectic} submanifold of $N$) and restricted Hamiltonian $H\big|_M = -\frac{32}{5}c^5 + \frac12 B(a,c,t)$, where
$$B(a,c,t) \defeq \int b(x,t) \eta(x,a,c)^2 \, dx \,.$$
The heuristic adopted in \cite{HZ1, HZ2}, essentially equivalent (see \cite{HZ3}) to the ``effective Lagrangian'' or ``collective coordinate method'' commonly applied in the physics literature, is the following:  the equations of motion for $a$, $c$ are approximately the Hamilton flow of $H\big|_M$ with respect to $\omega\big|_M$.  These equations are
$$
\left\{
\begin{aligned}
&\dot a = 4c^2 - \frac1{16}c^{-2} \partial_c B \\
&\dot c = \frac1{16}c^{-2}\partial_a B
\end{aligned}
\right.
$$
By Taylor expansion, these equations are approximately
$$
\left\{
\begin{aligned}
&\dot a = 4 c^2 - b(a) + O(h^2) \\
&\dot c = \frac13 c b'(a) + O(h^3)
\end{aligned}
\right.
$$
Note that the equations \eqref{E:ODEs-scaled} are the rescaled versions of these equations with the $O(h^2)$ and $O(h^3)$ error terms dropped. 

The first of the orthogonality conditions in \eqref{E:so} can be rewritten as $\omega(v, \partial_a \eta)=0$ and thus interpreted as symplectic orthogonality with respect to the $a$-direction on $M$.   The other symplectic orthogonality condition $0=\omega(v, \partial_c \eta) = \la v, \partial_x^{-1}\partial_c \eta\ra$  is not defined for general $H^1$ functions $v$ since $\partial_x^{-1}\partial_c \eta = (\tau(y)+y\theta(y))\big|_{y=c(x-a)}$, where $\tau(y)=2\tanh y$.  Thus, we drop this condition, although it must be replaced with some other condition that projects sufficiently far away from the kernel ($\spn \{ \partial_x \eta \}$) of the Hessian of the Lyapunov functional.  We select $\la v, (x-a) \eta \ra=0$ (i.e., the second equation in \eqref{E:so}) since it is a hypothesis in the Martel-Merle local virial identity (Lemma \ref{L:MM}).

\subsection{Numerics}

For the numerics, we restrict to time-independent potentials $b(x)=b_0(hx)$ and use the rescaled frame $X=hx$, $S=h^3t$, $V(X,S) = h^{-2}u(h^{-1}X,h^{-3}S)$, and $B(X) = h^{-2}b(h^{-1}X)=h^{-2}b_0(X)$.  Then $V$ solves the equation
$$\partial_S V = \partial_X(-\partial_X^2 V - 3V^2 +BV) \,,$$
with initial data $V_0(X) = \eta(X,A_0,C_0h^{-1})$.
Note that to examine the solution $u(x,t)$ on the time interval $0\leq t \leq Kh^{-1}$, we should examine $V(X,S)$ on the time interval $0\leq  S\leq Kh^2$.

As an example, we put $b_0(x)=8\sin x$ and take $A_0=2.5$, $C_0=1$ and $K=1$.  Then the width of the soliton is approximately the same width as the potential (when $h=1$), but note that the size of the potential is not small.  The results of numerical simulations for $h= 0.3,0.2,0.1$ are depicted in the Fig. \ref{F:1}.  There, plots are given depicting the rescaled solution $v(X,S)$ for each of these values of $h$.  In Fig. \ref{F:2}, we draw a comparison to the ODEs \eqref{E:ODEs-scaled}.    In each of the numerical simulations, we record the center of the soliton as $\tilde A_h(S)$ and the soliton scale as 
$$\tilde C_h(S) = \sqrt{ \frac{ \text{max. amp}(S)}{2}} \,.$$
That is, we fit the solution $V(X,S)$ to $\eta(X,\tilde A_h(S), \tilde C_h(S))$.
 Let $T=ht$ so that $S=h^2T$.  To convert into the $(X,T)$ frame of reference,  we plot $T$ versus $A_h(T)=\tilde A_h(h^{2}T)$ in the top plot of Fig. \ref{F:2} together with $A(T)$ solving \eqref{E:ODEs-scaled}.  In the bottom frame, we plot $T$ versus $C_h(T) = h \tilde C_h(h^{2}T)$ together with $C(T)$ solving \eqref{E:ODEs-scaled}.  We opted to only plot $h=0.2$ since the curves for $h=0.3,0.2,0.1$ were all rather close, producing a crowded figure.
Theorem \ref{T:main} predicts $O(h)$ convergence in both frames of Fig. \ref{F:2}.  

The numerical solution to the equation \eqref{E:pKdV} was produced using a \texttt{MATLAB} code based on the Fourier spectral/ETDRK4 scheme as presented in  Kassam-Trefethen \cite{KT}.  The ODEs \eqref{E:ODEs-scaled} were solved numerically using \texttt{ODE45} in \texttt{MATLAB}. 

\begin{figure}
\includegraphics[scale=0.55]{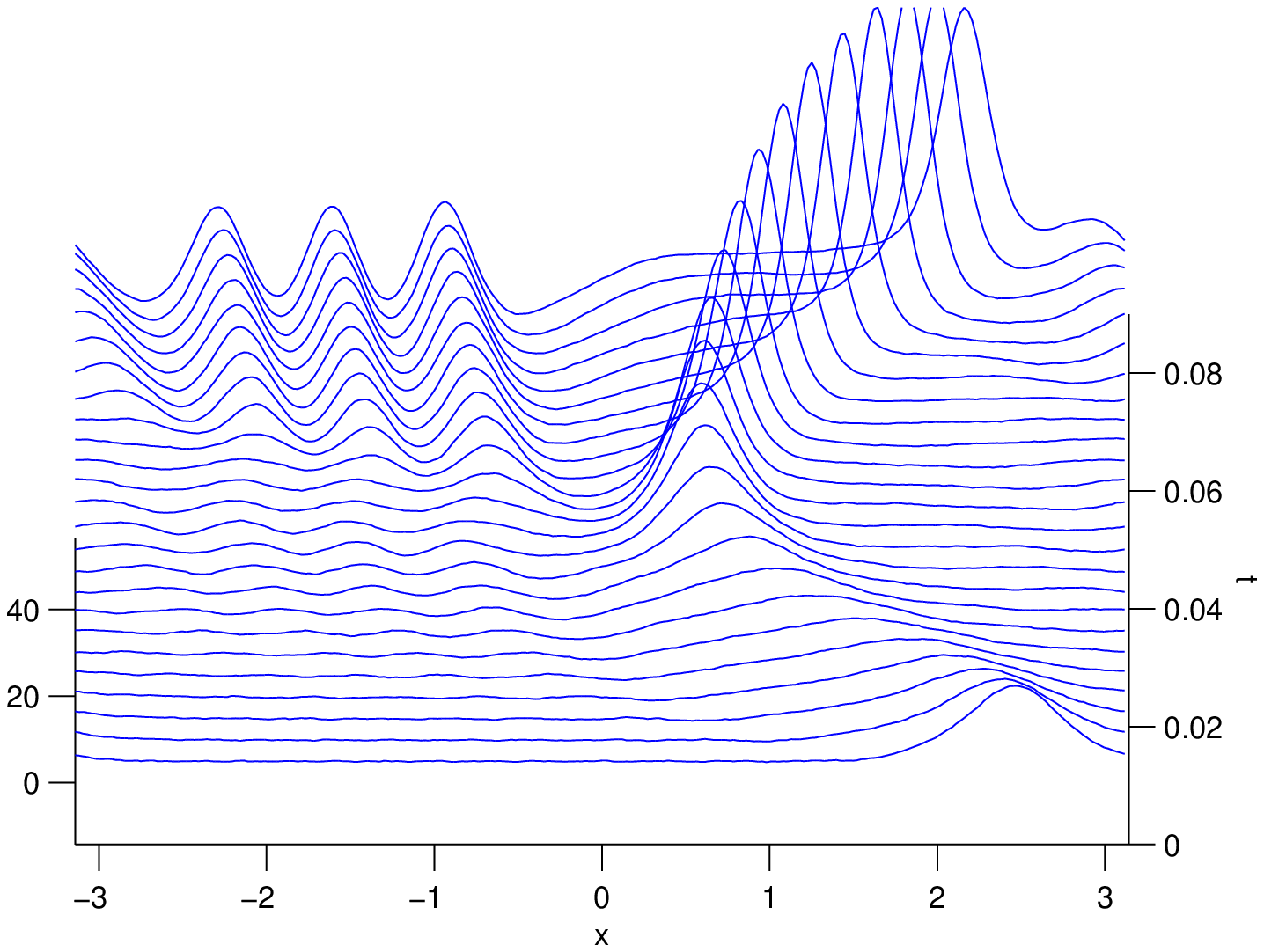}
\includegraphics[scale=0.55]{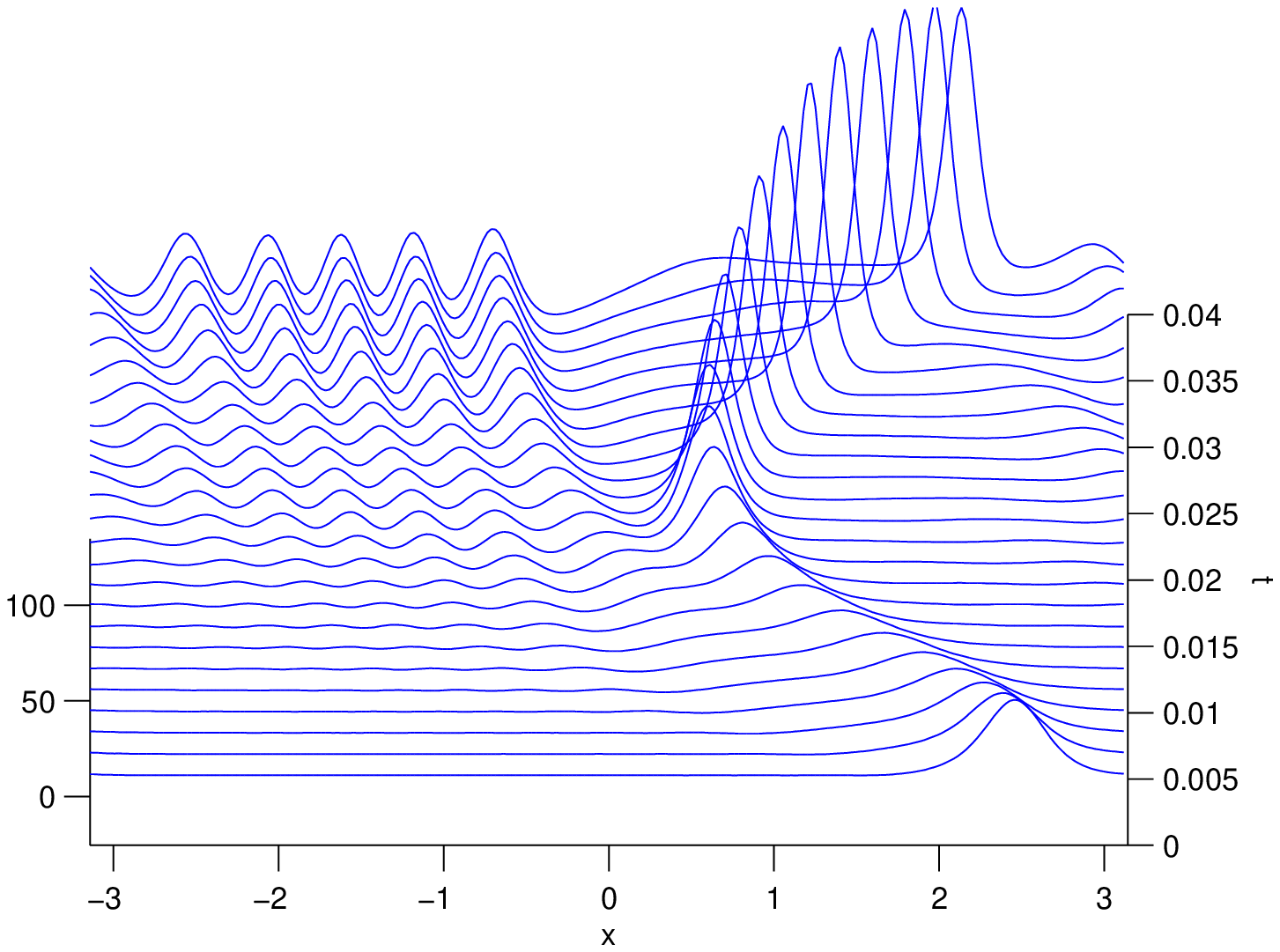}
\includegraphics[scale=0.55]{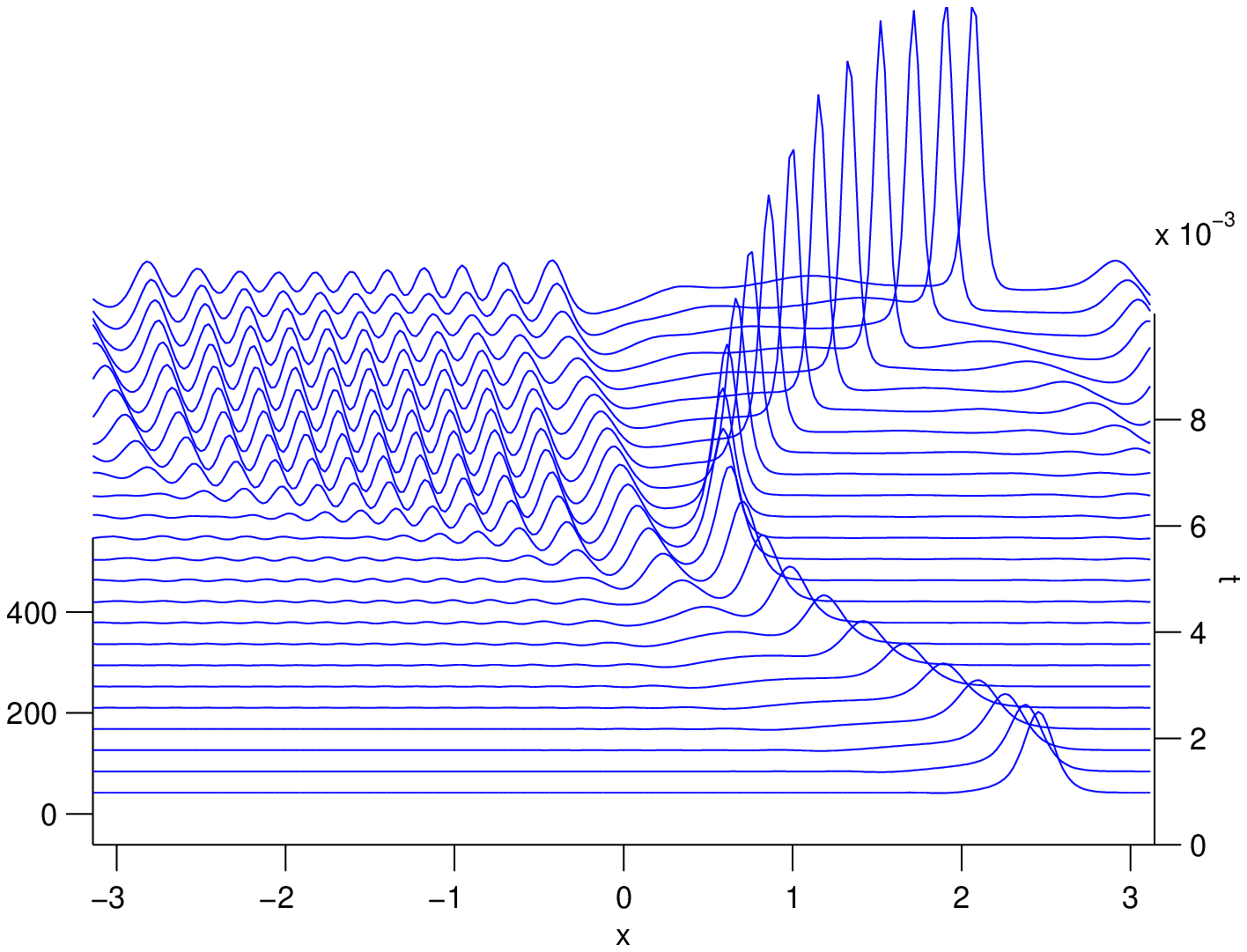}
\caption{The rescaled evolution $V(X,S)$ (see text) for $B(X)=h^{-2}b_0(X)=8h^{-2}\sin X$, $A_0=2.5$, $C_0=1$, on the time interval $0\leq S\leq h^2$.  The three frames are, respectively, $h=0.3$, $h=0.2$, and $h=0.1$.
\label{F:1}
}
\end{figure}

\begin{figure}
\includegraphics[scale=0.7]{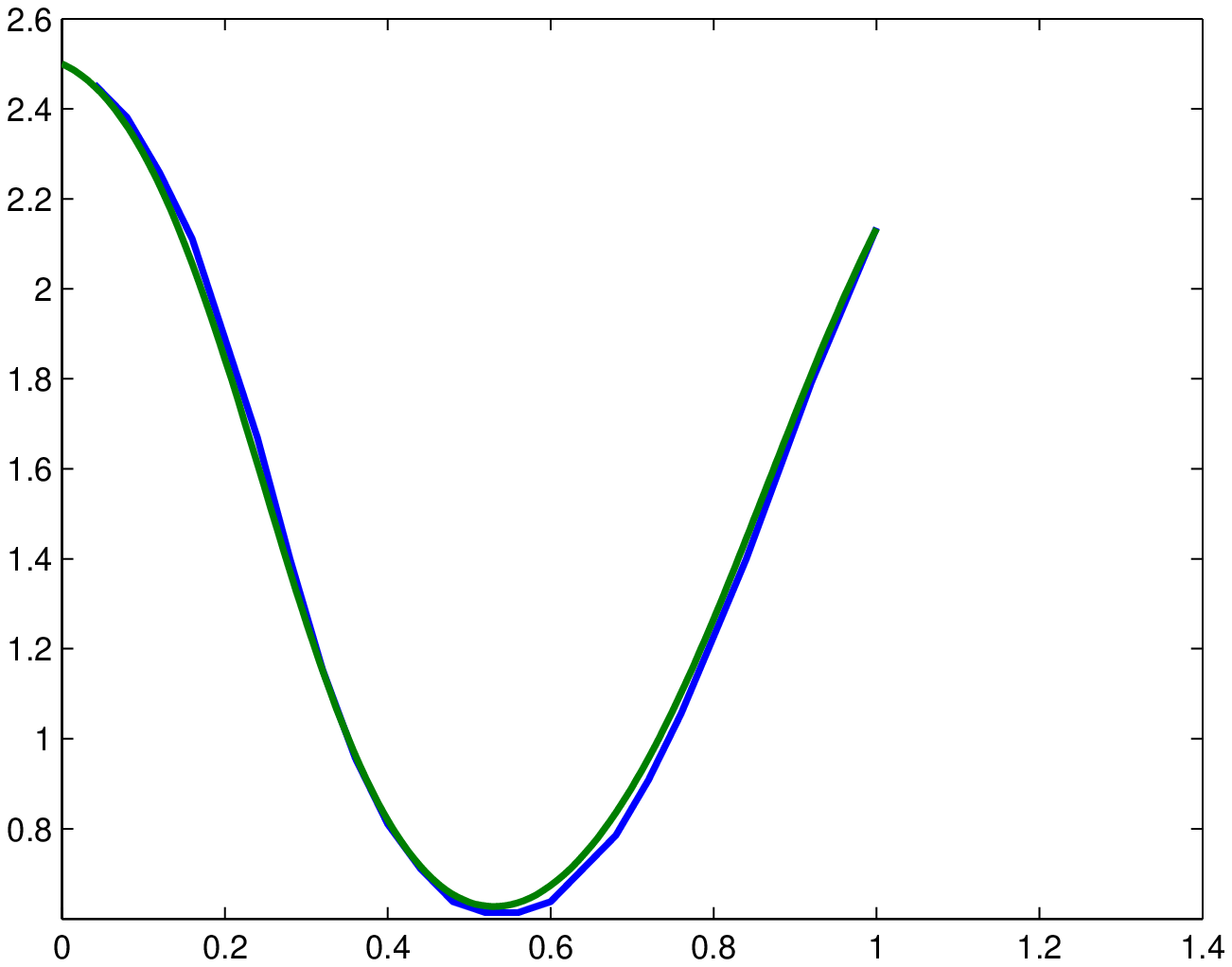}
\includegraphics[scale=0.7]{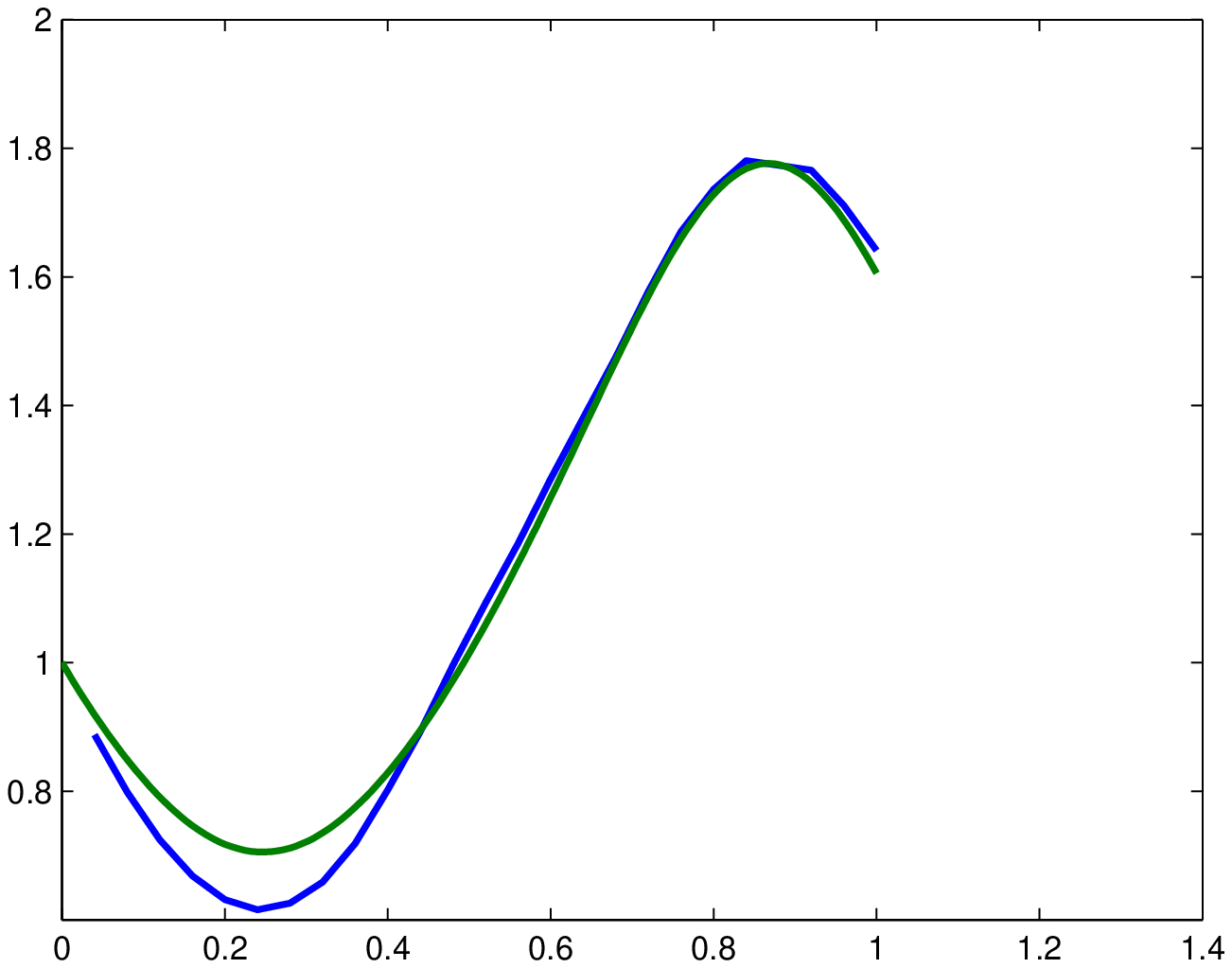}
\caption{For the simulations in Fig. \ref{F:1}, the position was recorded as $\tilde A_h(S)$ and the scale was recorded as $\tilde C_h(S)$; that is, the solution $v(X,S)$ was fitted to $\eta(X,\tilde A_h(S), \tilde C_h(S))$.
The top plot is $T$ versus $A_h(T)=\tilde A_h(h^2T)$ for $h=0.2$ (in blue) compared to the value of $A(T)$ obtained by solving the ODE system (in green).  
The bottom plot is $T$ versus $C_h(T)=h \tilde C_h(h^2T)$ for $h=0.2$ (in blue), compared to the value of $C(T)$ obtained from the ODE system (in green). 
\label{F:2}  
}
\end{figure}

\subsection{Relation to earlier and concurrent work}
\label{S:earlierwork}

Theorem 2 in Dejak-Sigal \cite{DS} states (roughly) that for potential $b(x,t) = \epsilon b(hx,ht)$, the error $\|w\|_{H^1} \lesssim \epsilon^{1/2}h^{1/2}$ can be achieved on the time-scale $t \lesssim (h+\epsilon^{1/2}h^{1/2})^{-1}$, and the equations of motion satisfy
$$
\left\{
\begin{aligned}
&\dot a = 4c^2-b(a) +O(\epsilon h)\\
&\dot c = O(\epsilon h)
\end{aligned}
\right.
$$
To reach the nontrivial dynamical time frame, one thus needs to take $\epsilon = h$ in their result.  With this selection for $\epsilon$, the $O(h^2)$ errors in the ODEs can be removed as in our result with the effect of at least preserving the error estimate for $w$ in $H^1$ at the $h^{1/2}$, rather than $h$ level.  But then the conclusion of their analysis is that the (small and slowly varying) potential has no significant effect on the dynamics.  We emphasize that in our case, we allow for $\epsilon = O(1)$ and thus can see dramatic effects on the motion of the soliton.

The paper Dejak-Sigal \cite{DS} is modeled upon earlier work by Fr\"ohlich-Gustafson-Jonsson-Sigal \cite{FGJS} for the NLS equation, which controlled the error via the Lyapunov functional employed in the orbital stability theory of Weinstein \cite{Wei}.  In \cite{HZ2}, we improved \cite{FGJS} by using the symplectic restriction interpretation as a guide in the analysis and introducing a correction term to the Lyapunov estimate.  A correction term is not as easily applied to the study of \eqref{E:pKdV} since the leading order inhomogeneity in the equation for $v$ generates a ``nonlocal'' solution.  To properly address the nonlocality of $v$, we use both the global $H^1$ estimate \eqref{E:v-global} as in \cite{FGJS, DS, HZ1, HZ2}, but also introduce the new local estimate \eqref{E:v-local}, which is proved using the local virial identity of Martel-Merle \cite{MM}.  We remark that our method does not use the integrable structure of the KdV equation, and we expect that our result will carry over to the perturbed $L^2$ subcritical gKdV-$p$ equation
 $$\partial_t u = - \partial_x (\partial_x^2 u + u^p - bu)$$
In the case $p=3$, i.e. the second symplectic orthogonality condition $\la v, \partial_x^{-1}\partial_c \eta\ra =0$ (where now $\theta(y) = \sqrt 2\sech y$ and $\eta(x,a,c)=c\eta(c(x-a))$) is well-defined for general $H^1$ functions $v$.  In this case, we are able to achieve stronger results by following the method of \cite{HZ2}, and even treat double solitons -- see \cite{HPZ}.

The concurrent work by Mu\~noz \cite{Munoz} considers the equation (specializing to the case $m=2$ in his paper to facilitate comparison)
\begin{equation}
\label{E:Munoz}
\partial_t v = \partial_x( -\partial_x^2 v + 4\lambda v - 3\alpha v^2) \,,
\end{equation}
where $\alpha(x)=\alpha_0(hx)$, with $\alpha_0(X)$ increasing monotonically from $\alpha(-\infty)=1$ to $\alpha(\infty)=2$, and $0\leq \lambda \leq \lambda_0 \defeq \frac35$ is constant, effectively corresponding to a moving frame of reference.  The equation \eqref{E:Munoz} is similar to our \eqref{E:pKdV} but not directly related to it through any known transformation.  His main theorem gives the existence of a solution $v(x,t)$  which asymptotically matches the soliton $\eta(x,a(t),1)$ as $t\to -\infty$ and matches the soliton $\frac12 \eta(x,a(t),c_\infty)$ as $t\to +\infty$ with error at most $h^{1/2}$ in $H^1_x$.  Here, $c_\infty$ is precisely given in terms of the solution to an algebraic equation (see (4.17) in his paper).  He presents this problem as more of an obstacle scattering problem with a careful analysis of ``incoming''  and ``outgoing'' waves and thus his priorities are different from ours.  

However, information from the ``interaction phase'' of his analysis can be extracted from the main body of his paper and compared with the results of our paper.  In the course of his analysis, he obtains effective dynamics (here $\lambda_0=\frac35$) for an \emph{approximate solution}
$$
\left\{
\begin{aligned}
&\dot a = c^2-\lambda \\
&\dot c = \frac25 c\left(c^2 - \frac{\lambda}{\lambda_0} \right) \frac{\alpha'(a)}{\alpha(a)} 
\end{aligned}
\right.
$$
He then shows that the approximate solution is comparable to a true solution in $H^1$ with accuracy $O(h^{1/2})$ (same as in our result) but only at the expense of a spatial shift for which he has the comparatively weak control of size $O(h^{-1})$.  In our analysis, we are able to achieve control of size $O(1)$ on the positional parameter $a(t)$.    At the technical level, we are gaining an advantage by using the local virial estimate in the interaction phase analysis while Mu\~noz carries out a more direct energy estimate.   Mu\~noz does apply the local virial estimate in his ``post-interaction'' analysis to achieve a convergence statement as $t\to +\infty$ with a remarkably precise scale estimate.

\subsection{Notation}

It is convenient to work in both direct (e.g. $\eta(x,a,c)$) and ``pulled-back'' coordinates (e.g $\theta(y)$).  Our convention is that successive letters are used to define functions related in this way.  Specifically,
\begin{itemize}
\item $\theta(y) = 2\sech y$ and $\eta(x,a,c) = c^2\theta(c(x-a))$.  
\item $\tau(y) = 2\tanh y$ and $\sigma(x,a,c) = c^2\tau(c(x-a))$.
\item $v(x,t) = 2c^2w(c(x-a),t)$
\item $\mathcal{L} = 4-\partial_y^2 - 6\theta$ and $\mathcal{K} = 4c^2- \partial_x^2 - 6\eta(\cdot,a,c)$.
\end{itemize}

\subsection{Outline of the paper}

In \S\ref{S:spectral}, we deduce some needed spectral properties of the operator $\mathcal{K}$ which are required to give the lower bound in the Lyapunov functional method (Cor. \ref{C:spec-bd}).  In \S\ref{S:orth}, we give the standard argument, via the implicit function theorem, that the parameters $a$ and $c$ can be adjusted so as to arrange that $v$ satisfies the orthogonality conditions \eqref{E:so} (Lemma \ref{L:ift}).  In \S\ref{S:decomposition}, we decompose the forcing term in the linearized equation into symplectically orthogonal and symplectically parallel components.  In \S\ref{S:parameters}, the orthogonality conditions are applied to obtain the equations for the parameters (Lemma \ref{L:param}).  These equations include error terms expressed in terms of the local-in-space norm $\|e^{-\epsilon|x-a|}v\|_{H^1}$.  In \S\ref{S:locvir}, an estimate on $\|e^{-\epsilon|x-a|}v\|_{L_T^2H^1}$ is obtained by the Martel-Merle local virial identity (Lemma \ref{L:locvir}).  In \S\ref{S:energy}, the estimates on $\|v\|_{L_T^\infty H_x^1}$ are obtained by the Lyapunov energy method (Lemma \ref{L:energy}).  The three key estimates (Lemmas \ref{L:param},  \ref{L:locvir}, \ref{L:energy}) are combined to give the proof of Theorem \ref{T:main} in \S\ref{S:proof}. 

\subsection{Acknowledgements}

Galina Perelman shared with me a set of notes illustrating how to apply the Martel-Merle local virial identity to this problem.  The present paper is essentially an elaboration of this note, and hence I am very much indebted to her generous assistance.  I thank also Maciej Zworski for initially proposing the problem, providing the numerical codes, and for helpful discussions.

I am partially supported by a Sloan fellowship and NSF grant DMS-0901582.

\section{Spectral properties of the linearized operator}
\label{S:spectral}

Recall that $\mathcal{L} = 4 - \partial_y^2 -6 \theta$.
Since $\theta(y) = 2\sech^2 y$, we see that we must consider the Schr\"odinger operator with P\"oschl-Teller potential 
$$A=-\partial_y^2 - \nu(\nu+1)\sech^2y$$
with $\nu =3$.  The spectral resolution of operators of the type $A$ is deduced via hypergeometric functions in the appendix of Guillop\'e-Zworski \cite{GZ}.  From this analysis, we obtain

\begin{lemma}[spectrum of $\mathcal{L}$]
The spectrum of $\mathcal{L}$ is $\{-5,0,3\} \cup [4,+\infty)$.  The $L^2$ normalized eigenfunctions corresponding to the first two eigenvalues are 
\begin{align*}
& \lambda_1 = -5  && f_1(y) = \frac{\sqrt{15}}{4}\sech^3y \\
& \lambda_0 = 0  && f_0(y) = \frac{\sqrt{15}}{2}\sech^2 y\tanh y=- \frac{\sqrt{15}}{8} \theta'(y)
\end{align*}
\end{lemma}
Denote by $E_j$ the corresponding eigenspaces and $P_{E_j}$ the corresponding projections (that is, the $L^2$ orthogonal projections and not the symplectic orthogonal projections).

\begin{lemma}
Suppose that $\la w, \theta\ra =0$ and $\la w, y\theta \ra =0$.  Then
\begin{equation}
\label{E:L2bd}
2\|w\|_{L^2}^2 \leq \la \mathcal{L}w,w\ra
\end{equation}
\end{lemma}
\begin{proof}
Since $\mathcal{L}$ preserves parity, it suffices to separately prove:

\noindent\textit{Claim 1}.  If $w$ is even, $\|w\|_{L^2}=1$, and $\la w, \theta\ra =0$, then $\la \mathcal{L}w, w\ra \geq 2$.

\noindent\textit{Claim 2}.  If $w$ is odd, $\|w\|_{L^2}=1$, and $\la w, y\theta \ra =0$, then $\la \mathcal{L}w, w \ra \geq 2$.

\smallskip

We begin with the proof of Claim 1.  Since $w$ is even, $\la w, f_0\ra=0$.  Resolve $w$ as
$$w=\alpha f_1+g, \quad g\in (E_1+E_0)^\perp \,, \quad \alpha^2+\|g\|_{L^2}^2=1 \,.$$  
Resolve also
$$\theta = \beta f_1 +h, \quad h\in (E_1+E_0)^\perp \,, \quad \beta^2+\|h\|_{L^2}^2=\|\theta\|_{L^2}^2=\frac{16}{3} \,.$$   
We compute that
\begin{equation}
\label{E:spec1}
\beta = \la \theta, f_1\ra = \frac{3\sqrt{15}\pi}{16} \approx 2.28138 \,,
\end{equation}
from which it follows that
\begin{equation}
\label{E:spec2}
\|h\|_{L^2}^2 = \frac{16}{3} - \left( \frac{3\sqrt{15}\pi}{16}\right)^2 \approx 0.128659 \,.
\end{equation}
We then have
$$0=\la w, \theta \ra = \alpha\beta + \la g, h \ra \,,$$
which using \eqref{E:spec1}, \eqref{E:spec2}, and $\|g\|_{L^2} \leq 1$, implies
$$|\alpha| \leq \frac{1}{\beta} \|g\|_{L^2}\|h\|_{L^2} \leq 0.157226 \,.$$
By the spectral theorem,
$$\la \mathcal{L}w,w\ra \geq 3 \|g\|_{L^2}^2 - 5\alpha^2 = 3(1-\alpha^2) - 5\alpha^2 = 3-8\alpha^2 \geq 2 \,.$$

Next, we prove Claim 2.  Since $w$ is odd, $\la w, f_1 \ra =0$.   Resolve $w$ as
$$w = \alpha f_0 + g \,, \quad g\in (E_1+E_0)^\perp \,, \quad \alpha^2 + \|g\|_{L^2}^2 = 1 \,.$$
Resolve also
$$y\theta = \beta f_0 + h \,, \qquad h \in (E_1+E_0)^\perp \,, \quad \beta^2 + \|h\|_{L^2}^2 = \|y\theta\|_{L^2}^2 = \frac{4}{9}(\pi^2-6) \,.$$
We compute that
\begin{equation}
\label{E:spec3}
\beta = \la y\theta, f_0 \ra =  \sqrt{ \frac{5}{3}} \approx 1.29099 \,,
\end{equation}
from which it follows that
\begin{equation}
\label{E:spec4}
\|h\|_{L^2}^2 = \frac{4}{9}(\pi^2-6)- \beta^2 \approx 0.0531575\,.
\end{equation}
We then have
$$0 = \la w, y\theta \ra = \alpha \beta + \la g, h \ra \,, $$
which, using \eqref{E:spec3}, \eqref{E:spec4}, and $\|g\|_{L^2} \leq 1$ implies
$$|\alpha| \leq \frac{1}{\beta} \|g\|_{L^2}\|h\|_{L^2} \leq 0.17859  \,.$$
By the spectral theorem,
$$\la \mathcal{L}w,w\ra \geq 3 \|g\|_{L^2}^2 = 3 -3\alpha^2 \geq 2 \,.$$
\end{proof}

\begin{corollary}
Suppose that 
\begin{equation}
\label{E:so-w}
\la w, \theta \ra =0 \quad \text{and} \quad \la w, y\theta \ra =0 \,.
\end{equation}
Then
$$\tfrac{2}{11}\|w\|_{H^1}^2  \leq  \la \mathcal{L}w, w \ra \,.$$
\end{corollary}
\begin{proof}
By integration by parts, 
$$ \la \mathcal{L} w, w \ra = 4\|w\|_{L^2}^2 + \|\partial_x w\|_{L^2}^2 - 6\int \theta w^2 $$
from which we obtain
$$\|\partial_x w \|_{L^2}^2 \leq \la \mathcal{L}w, w\ra + 8\|w\|_{L^2}^2$$
Adding to this estimate $\frac{9}{2} \times$ the estimate \eqref{E:L2bd}, we obtain the claim.
\end{proof}

Of course the above properties of $\mathcal{L}$ can be converted to properties of $\mathcal{K}$, where
$$\mathcal{K} = 4c^2 - \partial_x^2 - 6\eta(\cdot, a,c) \,,$$
by scaling and translation.  In particular, we have

\begin{corollary}
\label{C:spec-bd}
Suppose that 
\begin{equation}
\label{E:so-v}
\la v, \eta(\cdot,a,c) \ra =0 \quad \text{and} \quad \la v, (x-a) \eta(\cdot,a,c) \ra =0 \,.
\end{equation}
Then
$$\|v\|_{H^1}^2  \lesssim \la \mathcal{K} v, v \ra \,.$$
where the implicit constant depends on $c$.
\end{corollary}

\section{Orthogonality conditions}
\label{S:orth}

We next show by a standard argument that the parameters $(a,c)$ can be tweaked to achieve the orthogonality conditions \eqref{E:so}.

\begin{lemma}
\label{L:ift}
If $\delta \leq \tilde c \leq \delta^{-1}$, there exist constants $\epsilon>0$, $C>0$ such that the following holds.  If $u=\eta(\cdot, \tilde a, \tilde c) + \tilde v$ with $\|\tilde v\|_{H^1} \leq \epsilon$, then there exist unique $a$, $c$ such that 
$$|a-\tilde a| \leq C\|\tilde v\|_{H^1} \,, \quad |c - \tilde c| \leq C\|\tilde v\|_{H^1}$$
and $v \defeq u - \eta(\cdot,a,c)$ satisfies
$$\la v, \eta \ra =0 \quad \text{and} \quad \la v, (x-a) \eta \ra=0 \,.$$
\end{lemma}
\begin{proof}
Define a map $\Phi: H^1 \times \mathbb{R} \times \mathbb{R}^+ \to \mathbb{R}^2$ by
$$\Phi(u,a,c) = 
\begin{bmatrix}
\la u - \eta(\cdot,a,c) , \eta \ra \\
\la u - \eta(\cdot,a,c) , (x-a) \eta \ra
\end{bmatrix}
$$
The derivative of $\Phi$ with respect to $(a,c)$ at the point $(\eta(\cdot,\tilde a,\tilde c), \tilde a, \tilde c)$ is 
$$(D_{a,c} \Phi) (\eta(\cdot, \tilde a, \tilde c), \tilde a, \tilde c) = 
-\begin{bmatrix}
\la \partial_a\eta, \eta\ra 
& \la \partial_c \eta, \eta \ra \\
\la \partial_a\eta, (x-a) \eta \ra 
& \la \partial_c \eta, (x-a) \eta \ra
\end{bmatrix}
 =
\begin{bmatrix}
0 & 8c^2 \\
\frac{8}{3}c^3 & 0
\end{bmatrix} \,,
$$ 
which is nondegenerate.  By the implicit function theorem, the equation $\Phi(u,a,c)=0$ can be solved for $(a,c)$ in terms of $u$ in a neighborhood of $\eta(\cdot, \tilde a, \tilde c)$.
\end{proof}

\section{Decomposition of the flow}
\label{S:decomposition}

Since we will model $u=\eta(\cdot,a,c)+v$ and $u$ solves \eqref{E:pKdV}, we compute that $v$ solves
\begin{align}
\notag
\partial_t v &= -\partial_x(\partial_x^2 v + 6\eta v -b v + 3v^2) + F_0 \\
\label{E:v} 
&= \partial_x \mathcal{K}v - 4c^2 \partial_x v + \partial_x(bv) -3\partial_x v^2 +F_0
\end{align}
where
$$F_0 = - (\dot a - 4c^2)\partial_a \eta   - \dot c \partial_c \eta + \partial_x(b\eta) \,.$$
Decompose $F_0=F_\|+F_\perp$, where $F_\|$ is symplectically parallel to $M$ and $F_\perp$ is symplectically orthogonal to $M$.   Explicitly, we have
\begin{align*}
&F_\| = \left(-(\dot a-4c^2) - \frac1{16c^2} \partial_c B \right) \partial_a \eta + \left( -\dot c + \frac1{16c^2} \partial_a B \right) \partial_c \eta \\
&F_\perp = \frac1{16c^2} \partial_c B  \; \partial_a \eta - \frac1{16c^2} \partial_a B  \; \partial_c \eta + \partial_x(b\eta) 
\end{align*}
By Taylor expansion we obtain $F_\perp = (F_\perp)_0 +O(h^2)$, where
$$(F_\perp)_0 = \frac13 c^2 b'(a) \; (\theta(y)+2y\theta'(y))\big|_{y=c(x-a)} \,.$$
By definition of $F_\perp$, we have $\la F_\perp, \partial_x^{-1}\partial_a\eta\ra =0$ and $\la F_\perp, \partial_x^{-1}\partial_c\eta\ra =0$ , which must then hold at every order in $h$; in particular, they hold for $(F_\perp)_0$.  Note that by parity $(F_\perp)_0$ in addition satisfies $\la (F_\perp)_0, (x-a) \eta\ra=0$, although this is not expected to hold for $F_\perp$ at all orders.

It follows that
\begin{equation}
\label{E:F0-est}
\|e^{\epsilon |x-a|} F_0\|_{H_x^1} \lesssim |\dot a - 4c^2 - b(a)| + |\dot c - \tfrac13 cb'(a)| + h\,.
\end{equation}

\section{Equations for the parameters}
\label{S:parameters}

\begin{lemma}
\label{L:param}
Suppose that we are given $b_0\in C_c^\infty(\mathbb{R}^2)$ and $\delta>0$.  (Implicit constants below depend only on $b_0$ and $\delta$.)
Suppose that $\|v\|_{H_x^1} \ll 1$, $v$ solves \eqref{E:v} and satisfies \eqref{E:so}, and $\delta \leq c \leq \delta^{-1}$.  Then
\begin{equation}
\label{E:c}
|\dot c -\tfrac13cb'(a)| \lesssim h\|e^{-\epsilon |x-a|}v\|_{H^1} + \|e^{-\epsilon |x-a|}v\|_{H^1}^2 + h^2
\end{equation}
and
\begin{equation}
\label{E:a}
\left|\dot a - 4c^2 +b(a) + \frac{\la \partial_x\mathcal{K}v, (x-a)\eta\ra}{\la \partial_x\eta, (x-a)\eta\ra}\right| \lesssim h\|e^{-\epsilon |x-a|}v\|_{H^1}  +\|e^{-\epsilon |x-a|}v\|_{H^1}^2+ h^2 \,.
\end{equation}
\end{lemma}

\begin{proof}
We first work with the orthogonality condition $\la v, \partial_x^{-1}\partial_a \eta \ra =0$ to obtain \eqref{E:c}.  Applying $\partial_t$ to this orthogonality condition, we obtain
$$0 = \la \partial_tv , \eta(\cdot,a,c) \ra + \la v, \partial_t \eta(\cdot,a,c) \ra \,.$$
Substituting the equation for $v$ and the relation $\partial_t \eta = \dot a \partial_a \eta + \dot c \partial_c \eta$, we obtain
$$0 = 
\begin{aligned}[t]
&\la \partial_x \mathcal{K}v, \eta \ra  - 4c^2 \la \partial_x v, \eta \ra +\la \partial_x(bv), \eta\ra - 3\la \partial_x v^2, \eta \ra 
&& \leftarrow \text{I}+\text{II}+\text{III}+\text{IV}\\
&+ \la F_\|, \eta \ra + \la F_\perp, \eta \ra + \dot a \la v, \partial_a \eta \ra + \dot c \la v, \partial_c \eta \ra 
&& \leftarrow \text{V}+\text{VI}+\text{VII}+\text{VIII}
\end{aligned}
$$
We have $\text{I}=0$ and $\text{II}=0$.  Next, we calculate 
\begin{align*}
\text{III} &= \la \partial_x(bv),\eta\ra = - \la bv, \eta' \ra = b(a) \la v, \eta' \ra + O(h)\|e^{-\epsilon |x-a|} v\|_{H^1}\\
 &=  O(h)\|e^{-\epsilon |x-a|} v\|_{H^1}
\end{align*}
We easily obtain $|\text{IV}| \lesssim \|e^{-\epsilon |x-a|}v\|_{H_x^1}^2$.  Next,
\begin{align*}
\text{V} &= \la F_\|, \eta \ra = \left( -\dot c + \frac{1}{16c^2}\partial_a B\right) \la \partial_c \eta, \eta \ra \\
&= - \left( -\dot c + \frac{1}{16c^2}\partial_a B\right) \la \partial_c \eta, \partial_x^{-1}\partial_a \eta\ra \\
&= 8c^2 \left( - \dot c + \frac{1}{16c^2}\partial_a B \right) \,,
\end{align*}
from which it follows that
$$\text{V} = -8c^2(\dot c - \tfrac13cb'(a)) + O(h^2) \,.$$
Next, we have $\text{VI}=0$ and $\text{VII}=0$.  Finally,
$$|\text{VIII}| \lesssim |\dot c -\tfrac13 cb'(a)| \|e^{-\epsilon |x-a|}v\|_{H_x^1} + h\|e^{-\epsilon |x-a|}v\|_{H_x^1} \,.$$ 
Using that $\|v\|_{H_x^1}\ll 1$, we obtain \eqref{E:c}.  

To establish \eqref{E:a}, we apply $\partial_t$ to $\la v, (x-a) \eta \ra =0$ to obtain
$$0 = \la \partial_t v , (x-a) \eta \ra + \la v, \partial_t[(x-a) \eta] \ra$$
Substituting the equation \eqref{E:v} for $v$ and the relation $\partial_t \eta = \dot a \partial_a \eta + \dot c \partial_c \eta$, we obtain
$$0 = 
\begin{aligned}[t]
&\la \partial_x \mathcal{K}v, (x-a)\eta \ra  
- 4c^2 \la \partial_x v, (x-a)\eta \ra 
+ \la \partial_x(bv), (x-a)\eta\ra 
&& \leftarrow \text{I}+\text{II}+\text{III} \\
&- 3\la \partial_x v^2, (x-a)\eta \ra 
+ \la F_\|, (x-a)\eta \ra 
+ \la F_\perp, (x-a)\eta \ra 
&&\leftarrow \text{IV} + \text{V}+\text{VI} \\
&+ \dot a \la v, \partial_a [(x-a)\eta] \ra + \dot c \la v, (x-a)\partial_c \eta \ra 
&&\leftarrow \text{VII}+\text{VIII}
\end{aligned}
$$ 
Note that we do \emph{not} have $\text{I}=0$.  We would have $\text{I}=0$ if we were working with the orthogonality condition $\la v, \partial_x^{-1}\partial_c \eta \ra=0$, but as explained previously, this condition cannot be imposed on $v$ via the method of Lemma \ref{L:ift}, and even if it could, would not give the coercivity in Corollary \ref{C:spec-bd}.  We therefore keep Term I as is for now.  Next, we note that
$$\text{II} + \text{III} + \text{VII} 
= (-4c^2 + b(a)+\dot a) \la \partial_x v, (x-a)\eta \ra 
+ O(h) \|e^{-\epsilon |x-a|}v\|_{H_x^1} \,.$$
Next, $|\text{IV}| \lesssim \|e^{-\epsilon |x-a|} v\|_{H_x^1}^2$.  Also,
$$\text{V} =  \left( -\dot a + 4c^2 - \frac{1}{16c^2} \partial_c B \right) \la \partial_a\eta, (x-a) \eta \ra \,.$$
It happens that $\la \theta + 2y\theta', y\theta\ra=0$ and hence $\text{VI} = O(h^2)$.  Finally, $|\text{VIII}| \lesssim |\dot c| \|e^{-\epsilon |x-a|} v\|_{H_x^1} $, to which we can append the estimate \eqref{E:c}.  Collecting, we obtain \eqref{E:a}.  

\end{proof}

\section{Local virial estimate}
\label{S:locvir}

Next, we begin to implement the Martel-Merle \cite{MM} virial identity. Let $\Phi \in C(\mathbb{R})$, $\Phi(x)=\Phi(-x)$, $\Phi'(x) \leq 0$ on $(0,+\infty)$ such that $\Phi(x) = 1$ on $[0,1]$ and $\Phi(x) = e^{-x}$ on $[2,+\infty)$, and $e^{-x} \leq \Phi(x) \leq 3e^{-x}$ on $(0,+\infty)$.  Let $\tilde \Psi(x) = \int_0^x \Phi(y) \, dy$, and for $A\gg 1$ (to be chosen later) set $\psi(x) = A\tilde \Psi(x/A)$.

The following is the (scaled-out to unity version of) Martel-Merle's virial estimate.

\begin{lemma}[Martel-Merle {\cite[Lemma 1, Step 2 in Apx. B]{MM}} and {\cite[Prop. 6]{MM2}}]
\label{L:MM}
There exists $A$ sufficiently large and $\lambda_0>0$ sufficiently small such that if $w$ satisfies the orthogonality conditions
$$\la w, \theta \ra =0 \quad \text{and} \quad \la w, y\theta\ra =0 \,,$$
then we have the estimate
$$\lambda_0 \int (w_y^2 + w^2) e^{-|y|/A} \leq -\la  \psi w , \, \partial_y \mathcal{L}w \ra + \frac{\la \psi w, \theta'\ra \la \partial_y\mathcal{L}w,y\theta\ra}{\la\theta',y\theta\ra}\,.$$
\end{lemma}

Step 2 in Apx. B of \cite{MM} is a localization argument that shows that it suffices to consider the case $A=\infty$ and $\psi(y)=y$.  Some integration by parts manipulations and the fact that $\la w, \theta\ra =0$ convert this case to the estimate
\begin{equation}
\label{E:MMspec}
\|w\|_{H^1}^2 \lesssim \frac32 \la Lw,w \ra + \frac{6}{\|\theta\|_{L^2}^2}\la w, y\theta'\ra \la w, \theta^2 \ra \,,
\end{equation}
where $L = (\frac43 + 2y\theta' -2 \theta) - \partial_y^2$.  The positivity estimate \eqref{E:MMspec} appears as Prop. 3 in \cite{MM} and as Prop. 6 in \cite{MM2}, and is proved in \cite{MM2}.  

By scaling Lemma \ref{L:MM}, we obtain the following version adapted to $\mathcal{K}$.
\begin{corollary}
\label{C:MM}
There exists $A$ sufficiently large and $\lambda_0>0$ sufficiently small such that if $v$ satisfies the orthogonality conditions \eqref{E:so}, then  (with $\psi=\psi(x-a)$)
$$\lambda_0 \int (v_x^2 + v^2) e^{-c|x-a|/A} \leq 
- \la \psi v , \partial_x \mathcal{K}v \ra 
+ \frac{ \la \psi v, \partial_x\eta\ra \la \partial_x \mathcal{K}v, (x-a)\eta\ra}{\la \partial_x \eta, (x-a)\eta\ra}
\,.$$ 
\end{corollary}

\begin{lemma}[application of local virial identity]
\label{L:locvir}
Suppose that we are given $b_0\in C_c^\infty(\mathbb{R}^2)$ and $\delta>0$.  (Implicit constants below depend only on $b_0$ and $\delta$.)
Suppose that $|-\dot a + 4c^2 - b(a)| \ll 1$, $\|v\|_{H_x^1} \ll 1$, $v$ solves \eqref{E:v} and satisfies \eqref{E:so}, and $\delta\leq c \leq \delta^{-1}$.   Then with $\psi=\psi(x-a)$, we have 
\begin{equation}
\label{E:locvirest}
\| e^{-\epsilon |x-a|} v\|_{H_x^1}^2 \leq -\kappa_1 \partial_t \int \psi v^2 + \kappa_2 h^2 + \kappa_2 h\|v\|_{H_x^1}^2\,,
\end{equation}
where $\epsilon=\epsilon(\delta)>0$ and $\kappa_j=\kappa_j(\delta,b_0)>0$.
Integrating over $[0,T]$, we obtain with $T\lesssim h^{-1}$,
\begin{equation}
\label{E:new-11}
\|e^{-\epsilon |x-a|} v\|_{L_{[0,T]}^2H_x^1} \lesssim \|v\|_{L_{[0,T]}^\infty H_x^1} + T^{1/2}h\,.
\end{equation}
\end{lemma}
\begin{proof}
Recalling that $\psi = \psi(x-a)$,
$$
\partial_t \int \psi v^2 = 
\begin{aligned}[t]
&- \dot a \int \psi' v^2 + 2\int \psi \, v \, \partial_x \mathcal{K}v - 8c^2 \int \psi v \partial_x v + 2\int \psi v \partial_x(bv) \\
&- 6 \int \psi v \partial_x (v^2) + 2\int \psi v F_0 
\end{aligned}
$$
We reorganize the terms in the equation to
$$
\begin{aligned}[t]
\indentalign 
\underbrace{-2\int \psi v \, \partial_x \mathcal{K} v}_{\text{A}} 
\underbrace{-2\int \psi v F_0}_{\text{B}} \\
&= 
\underbrace{-\partial_t \int \psi v^2}_{\text{I}}
\underbrace{- \dot a \int \psi' v^2}_{\text{II}} 
\underbrace{- 8c^2 \int \psi v \partial_x v}_{\text{III}}
\underbrace{+ 2\int \psi v \partial_x(bv)}_{\text{IV}}
\underbrace{- 6 \int \psi v \partial_x(v^2)}_{\text{V}} \,.
\end{aligned}
$$
Note that we have written this equation symbolically in the form
\begin{equation}
\label{E:new-0}
\text{A}+\text{B} = \text{I}+\text{II}+\text{III}+\text{IV}+\text{V} \,,
\end{equation}
and we now consider these terms separately.  Integration by parts yields
\begin{align*}
\text{III} &= 4c^2 \int \psi' v^2\\
\text{IV} &= -\int \psi' b v^2 + \int \psi b_x v^2\\
& = -\int \psi' b(a)v^2 - \int \psi' (b(x)-b(a)) v^2 + \int \psi b_x v^2
\end{align*}
Hence
$$\text{II}+\text{III}+\text{IV} = (-\dot a + 4c^2 -b(a)) \int \psi' v^2 + O(h) \|v\|_{L^2}^2 \,,$$
from which it follows that
\begin{equation}
\label{E:new-1}
| \text{II}+\text{III}+\text{IV}| \lesssim |\dot a - 4c^2 + b(a)| \|e^{-\epsilon |x-a|} v\|_{L_x^2}^2 + h \|v\|_{L_x^2}^2 \,.
\end{equation}
Integration by parts also yields
$$\text{V} = 4 \int \psi' v^3 \,,$$
from which it follows that
\begin{equation}
\label{E:new-2}
|\text{V}| \lesssim  \|e^{-\epsilon|x-a|} v\|_{L_x^2}^2 \|v\|_{H_x^1} \,.
\end{equation}
Using that
$$F_0 = (\dot a- 4c^2 + b(a)) \partial_x \eta + O(h+|\dot c|) e^{-2\epsilon |x-a|} \,,$$
we obtain
$$\text{B} = -2 (\dot a -4c^2 + b(a))\la \psi v,  \partial_x \eta \ra + O(h+|\dot c|) \|e^{-\epsilon |x-a|} v \|_{L_x^2} \,.$$
By \eqref{E:a},
\begin{equation}
\label{E:new-3}
\text{B} = 2 \frac{\la \partial_x \mathcal{K}v, (x-a) \eta \ra}{\la \partial_x\eta, (x-a)\eta\ra} \la \psi v, \partial_x \eta \ra + O(h+|\dot c|) \|e^{-\epsilon |x-a|} v \|_{L_x^2} \,.
\end{equation}
Placing estimates \eqref{E:new-1}, \eqref{E:new-2}, and \eqref{E:new-3} into \eqref{E:new-0}, we obtain, for some constant $\kappa>0$, the bound 
\begin{align*}
\indentalign - 2\la \psi v, \partial_x \mathcal{K}v \ra + 2 \frac{ \la \psi v, \partial_x \eta\ra \la \partial_x \mathcal{K} v, (x-a) \eta \ra}{\la \partial_x \eta, (x-a) \eta \ra} \\
&\leq
\begin{aligned}[t]
&- \partial_t \int \psi v^2 + \kappa (|\dot a - 4c^2 + b(a)|+ \|v\|_{H_x^1}) \| e^{-\epsilon |x-a|} v\|_{L_x^2}^2   \\
&+ \kappa (h+|\dot c|) \|e^{-\epsilon |x-a|} v \|_{L_x^2} + \kappa h \|v\|_{L_x^2}^2
\end{aligned}
\end{align*}
Using  Corollary \ref{C:MM} and the assumptions $|\dot a - 4c^2 + b(a)| \ll 1$, $\|v\|_{H_x^1} \ll 1$, we obtain, for some constants $\kappa_1, \kappa_2>0$, the bound
\begin{equation}
\label{E:new-12}
\| e^{-\epsilon |x-a|}v\|_{H_x^1}^2 \leq - \kappa_1 \partial_t \int \psi v^2 + \kappa_2 (h+|\dot c|)^2 + \kappa_2 h \|v\|_{H_x^1}^2 \,.
\end{equation}
Note that \eqref{E:c} implies $|\dot c| \lesssim h + \|e^{-\epsilon|x-a|}v\|_{H_x^1}^2$.  Substituting this into \eqref{E:new-12} yields \eqref{E:locvirest}.
\end{proof}

\section{Energy estimate}
\label{S:energy}

Recall that $\mathcal{K} = 4c^2 - \partial_x^2 - 6\eta(\cdot,a,c)$.  Let 
$$\mathcal{E}(v) = \frac12 \la \mathcal{K}v,v\ra - \int v^3$$

\begin{lemma}[energy estimate]
\label{L:energy}
Suppose that we are given $b_0\in C_c^\infty(\mathbb{R}^2)$ and $\delta>0$.  (Implicit constants below depend only on $b_0$ and $\delta$.)
Suppose $v$ solves \eqref{E:v} and satisfies \eqref{E:so}, and $\delta \leq c\leq \delta^{-1}$.  Then
\begin{equation}
\label{E:energy-bd}
|\partial_t \mathcal{E}| \lesssim  
\begin{aligned}[t]
&|-\dot a + 4c^2 -b(a)| \|e^{-\epsilon |x-a|}v\|_{H_x^1}^2 + h\|v\|_{H_x^1}^2  \\
&+ h\|e^{-\epsilon |x-a|}v\|_{H^1} + \|e^{-\epsilon|x-a|}v\|_{H_x^1}^2 \|v\|_{H_x^1}^2
\end{aligned}
\end{equation}
\end{lemma}
We remark that by integrating \eqref{E:energy-bd} over $[0,T]$, $1\lesssim T \ll h^{-1}$, and applying Corollary \ref{C:spec-bd}, we obtain
\begin{equation}
\label{E:energy-bd2}
\|v\|_{L_{[0,T]}^\infty H_x^1} \lesssim 
\begin{aligned}[t]
&\|v_0\|_{H_x^1} + \| \dot a -4c^2 + b(a)\|_{L_{[0,T]}^\infty}^{1/2} \| e^{-\epsilon |x-a|}v\|_{L_{[0,T]}^2 H_x^1} \\
&  +  T^{1/4}h^{1/2} \|e^{-\epsilon |x-a|} v\|_{L_{[0,T]}^2H_x^1}^{1/2} + \|e^{-\epsilon|x-a|}v\|_{L_{[0,T]}^2H_x^1} \|v\|_{L_{[0,T]}^\infty H_x^1}\,.
\end{aligned}
\end{equation}

\begin{proof}
We compute
\begin{align*}
\partial_t \mathcal{E}(v) &=  \la \mathcal{K}v, \partial_t v\ra - 3\la v^2, \partial_t v\ra + 4c\dot c \|v\|_{L_x^2}^2 - 3 \la (\dot a \partial_a \eta  + \dot c\partial_c \eta) v,v\ra \\
&= \text{I}+\text{II}+\text{III}+\text{IV}
\end{align*}
Into I, we substitute \eqref{E:v}.
This gives
\begin{align*}
\text{I} &= \la \mathcal{K}v,\partial_x \mathcal{K}v \ra - 4c^2 \la \mathcal{K}v, \partial_x v\ra + \la \mathcal{K}v, \partial_x(bv) \ra - 3\la \mathcal{K}v, \partial_x v^2\ra + \la \mathcal{K}v, F_0 \ra \\
&= \text{IA}+\text{IB}+\text{IC}+\text{ID}+\text{IE}
\end{align*}
We have $\text{IA}=0$, while $\text{IB}= -12c^2 \la \eta_x, v^2\ra$.  For $\text{IC}$, numerous applications of integration by parts gives
$$\text{IC}= 2c^2 \la b_x, v^2 \ra +\frac32\la b_x, v_x^2\ra - \frac12 \la b_{xxx},v^2 \ra-3\la \eta b_x, v^2 \ra + 3\la \eta_x b, v^2 \ra \,,$$
and hence
$$\text{IC}= 3b(a)\la \eta_x, v^2 \ra +O( h \|v\|_{H^1}^2) \,.$$
Note
$$\text{IE} = \la v, \mathcal{K} F_\| \ra + \la v, \mathcal{K} F_\perp \ra \,.$$
But since $\mathcal{K}\partial_a \eta =0$, $\mathcal{K} \partial_c\eta = \eta$, and $\la v, \eta \ra=0$, we have $\la v, \mathcal{K} F_\| \ra =0$.  We estimate the second term to obtain
$$|\text{IE}| \lesssim h \|e^{-\epsilon |x-a|} v\|_{H_x^1} \,.$$ 
Combining, we obtain
$$\text{I} = 
\begin{aligned}[t]
&(12c^2 -3b(a))\la \partial_a\eta, v^2\ra -3\la \mathcal{K}v, \partial_x v^2 \ra \\
&+ O( h\|v\|_{H^1}^2 + h\|e^{-\epsilon|x-a|}v\|_{H^1}) \,.
\end{aligned}
$$ 
Substituting \eqref{E:v} into II, we obtain:
\begin{align*}
\text{II} &= -3\la v^2, \partial_x \mathcal{K}v\ra + 12 c^2 \la v^2, \partial_x v\ra - 3\la v^2, \partial_x(bv)\ra + 9 \la v^2, \partial_x v^2\ra - 3\la v^2, F_0\ra
\end{align*}
In II, we keep only the first term and estimate the rest to obtain
$$\text{II} = -3\la v^2, \partial_x \mathcal{K}v\ra + O ( h \|v\|_{H^1}^3 + \|e^{-\epsilon |x-a|}v\|_{H^1}^2 \|F_0 e^{2\epsilon |x-a|} \|_{H_x^1}) \,.$$
Note
$$\|e^{2\epsilon |x-a|} F_0\|_{H_x^1} \lesssim |\dot a- 4c^2-b(a)| + |\dot c| + h \,.$$
Collecting, we obtain 
\begin{equation}
\label{E:new-15}
|\partial_t \mathcal{E}| \lesssim  
\begin{aligned}[t]
&|-\dot a + 4c^2 -b(a)| \|e^{-\epsilon |x-a|}v\|_{H_x^1}^2 + (h+|\dot c|)\|v\|_{H_x^1}^2  \\
&+ h\|e^{-\epsilon |x-a|}v\|_{H^1} 
\end{aligned}
\end{equation}
Note that in the addition of terms I and II, the terms $\pm \la v^2, \partial_x \mathcal{K}v\ra$ canceled, and in the addition of I and IV, the two $O(1)$ coefficients $-3\dot a$ and $12c^2 -3b(a)$ were combined to give the smaller coefficient $-3\dot a + 12c^2 -3b(a)$.

Finally, we note that \eqref{E:c} implies $|\dot c|\lesssim h + \|e^{-\epsilon|x-a}v\|_{H_x^1}^2$.  Substituting this into \eqref{E:new-15} yileds \eqref{E:energy-bd}. 
\end{proof}

\section{Proof of Theorem \ref{T:main}}
\label{S:proof}

It will be shown later that Theorem \ref{T:main} follows from the following proposition.

\begin{proposition}
\label{P:main}
Suppose we are given $b_0\in C_c^\infty(\mathbb{R}^2)$ and $\delta>0$.  (Implicit constants below depend only on $b_0$ and $\delta$).   Suppose that we are further given $a_0\in \mathbb{R}$, $c_0>0$, $\kappa \geq 1$, $h>0$, and $v_0$ satisfying \eqref{E:so}, such that
$$0< h \lesssim \kappa^{-4} \,, \qquad \|v_0\|_{H_x^1} \leq \kappa h^{1/2} \,.$$  Let $u(t)$ be the solution to \eqref{E:pKdV} with $b(x,t)=b_0(hx,ht)$ and initial data $\eta(\cdot, a_0,c_0)+v_0$.   Then there exist a time $T'>0$ and trajectories $a(t)$ and $c(t)$ defined on $[0,T']$ such that $a(0)=a_0$, $c(0)=c_0$ and the following holds, with $v\defeq u-\eta(\cdot, a,c)$:
\begin{enumerate}
\item \label{I:orth} 
On $[0,T']$, the orthogonality conditions \eqref{E:so} hold.
\item \label{I:time-scale}
Either $c(T')=\delta$, $c(T')=\delta^{-1}$, or $T' \sim  h^{-1}$.
\item \label{I:en-bd}
$\|v\|_{L_{[0,T']}^\infty H_x^1} \lesssim \kappa h^{1/2} \,,$
\item \label{I:local-est}
$\|e^{-\epsilon |x-a|} v\|_{L_{[0,T']}^2 H_x^1}  \lesssim \kappa h^{1/2}$.
\item \label{I:dota}
$\int_0^{T'}|\dot a - 4c^2 +b(a)|\, dt \lesssim \kappa$.
\item \label{I:dotc}
$\int_0^{T'} |\dot c - \frac13 cb'(a) | \, dt \lesssim \kappa^2h$.
\end{enumerate}
\end{proposition}

\begin{proof}
Recall our convention that implicit constants depend only on $b_0$ and $\delta$.
By Lemma \ref{L:ift} and the continuity of the flow $u(t)$ in $H^1$, there exists some $T''>0$ on which $a(t)$, $c(t)$ can be defined so that \eqref{E:so} holds.  Now take $T''$ to be the maximal time on which $a(t)$, $c(t)$ can be defined so that \eqref{E:so} holds.  Let $T'$ be first time $0\leq T' \leq T''$ such that $c(T')=\delta$,  $c(T')=\delta^{-1}$, $T'=T''$, or $\omega h^{-1}$ (whichever comes first).  Here, $0 < \omega \ll 1$ is a constant that will be chosen suitably small at the end of the proof (depending only upon implicit constants in the estimates, and hence only on $b_0$ and $\delta$).

\begin{remark}
\label{R:bs1}
We will show that on $[0,T']$, we have $\|v(t)\|_{H_x^1}\lesssim \kappa h^{1/2}$, and hence by Lemma \ref{L:ift} and the continuity of the $u(t)$ flow, it must be the case that either  $c(T')=\delta$,  $c(T')=\delta^{-1}$, or $\omega h^{-1}$ (i.e. the case $T'=T''$ does not arise).
\end{remark}

Let $T$, $0<T\leq T'$, be the maximal time such that
\begin{equation}
\label{E:bs1}
\|v\|_{L_{[0,T]}^\infty H_x^1} \leq \alpha \kappa h^{1/2} \,,
\end{equation}
where $\alpha$ is a suitably large constant related to the implicit constants in the estimates (and thus dependent only upon $b_0$ and $\delta>0$).  In fact $\alpha \geq 1$ is taken to be $4$ times the implicit constant in front of $\|v_0\|_{H_x^1}$ in the energy estimate \eqref{E:energy-bd2}.

\begin{remark}
\label{R:bs2}
We will show, assuming that \eqref{E:bs1}  holds, that $\|v\|_{L_{[0,T]}^\infty H_x^1} \leq \frac12 \alpha \kappa h^{1/2}$ and thus by continuity we must have $T = T'$.
\end{remark}

In the remainder of the proof, we work on the time interval $[0,T]$, and we are able to assume that the orthogonality conditions \eqref{E:so} hold,  $\delta \leq c(t) \leq \delta^{-1}$, and that \eqref{E:bs1} holds.    We supress the $\alpha$ dependence in the estimates in  \eqref{E:a''} and \eqref{E:new-10} below.

By Lemma \ref{L:param},  \eqref{E:a}, and \eqref{E:bs1}, (just using that $\|e^{-\epsilon |x-a|}v\|_{H_x^1} \leq \|v\|_{H_x^1}$) it follows that
\begin{equation}
\label{E:a''}
|\dot a - 4c + b(a)| \lesssim \kappa h^{1/2} \,.
\end{equation}
By \eqref{E:a''}, the hypothesis of the local virial estimate Lemma \ref{L:locvir} is satisfied.  Using \eqref{E:bs1} in \eqref{E:new-11} (recall $T=\omega h^{-1} \leq h^{-1}$), we obtain
\begin{equation}
\label{E:new-10}
\|e^{-\epsilon |x-a|} v\|_{L_{[0,T]}^2 H_x^1} \lesssim \kappa h^{1/2} \,.
\end{equation}
Inserting \eqref{E:bs1},  \eqref{E:a''}, and \eqref{E:new-10} into the energy estimate \eqref{E:energy-bd2} (recall $T=\omega h^{-1}$),
we obtain
$$\|v\|_{L_{[0,T]}^\infty H_x^1} \leq \frac{\alpha}{4} \|v_0\|_{H_x^1} + C_\alpha (\kappa^{1/2}h^{1/4}+ \kappa h^{1/2} + \omega^{1/4}) \kappa h^{1/2}$$
Provided $h\lesssim_\alpha \kappa^{-2}$ and $\omega \ll_\alpha 1$, we obtain (recall $\|v_0\|_{H_x^1} \leq \kappa h^{1/2}$), we conclude that $\|v\|_{L_{[0,T]}^\infty H_x^1} \leq \frac12 \alpha \kappa^2 h$, completing the bootstrap, and demonstrating that $T=T'$.  In particular, we have established items \eqref{I:orth}, \eqref{I:time-scale}, \eqref{I:en-bd}, \eqref{I:local-est} in the proposition statement.  

It remains to prove \eqref{I:dota} and \eqref{I:dotc}.  
By Lemma \ref{L:param} \eqref{E:c},
\begin{align}
\notag
\int_0^T | \dot c - \tfrac13 cb'(a)| \, dt &\lesssim h T^{1/2} \|e^{-\epsilon |x-a|} v\|_{L_{[0,T]}^2 H_x^1} + \|e^{-\epsilon |x-a|} v\|_{L_{[0,T]}^2 H_x^1}^2 + Th^2\\
\notag
&\lesssim h T^{1/2} \kappa h^{1/2} + \kappa^2 h +Th^2\\
\label{E:c'''} 
&\lesssim \kappa^2 h \,,
\end{align}
establishing item \eqref{I:dotc}.
Similarly by Lemma \ref{L:param} \eqref{E:a}, we obtain item \eqref{I:dota}.
\end{proof}

The above proposition can be iterated to obtain:

\begin{corollary}
\label{C:main}
Suppose we are given $b_0\in C_c^\infty(\mathbb{R}^2)$ and $\delta>0$.  (Implicit constants and the constant $C$ below depend only on $b_0$ and $\delta$).   Suppose that we are further given $a_0\in \mathbb{R}$, $c_0>0$, $\beta \geq 1$, $h>0$, and $v_0$ satisfying \eqref{E:so}, such that
$$0< h \lesssim \beta^{-8} \,, \qquad \|v_0\|_{H_x^1} \leq \beta h^{1/2} \,.$$  
Let $u(t)$ be the solution to \eqref{E:pKdV} with $b(x,t)=b_0(hx,ht)$ and initial data $\eta(\cdot, a_0,c_0)+v_0$.   Then there exist a time $T'>0$ and trajectories $a(t)$ and $c(t)$ defined on $[0,T']$ such that $a(0)=a_0$, $c(0)=c_0$ and the following holds, with $v\defeq u-\eta(\cdot, a,c)$:
\begin{enumerate}
\item On $[0,T']$, the orthogonality conditions \eqref{E:so} hold.
\item Either $c(T')=\delta$, $c(T')=\delta^{-1}$, or $T' \sim  h^{-1}\log h^{-1}$.
\item $\|v\|_{L_{[0,T']}^\infty H_x^1} \lesssim \beta h^{1/2} e^{Cht} \,,$
\item $\|e^{-\epsilon |x-a|} v\|_{L_{[0,T']}^2 H_x^1}  \lesssim \beta h^{1/2} e^{Cht}$.
\item $\int_0^{T'}|\dot a - 4c^2 +b(a)|\, dt \lesssim \beta e^{Cht}$.
\item $\int_0^{T'} |\dot c - \frac13 cb'(a) | \, dt \lesssim \beta^2h e^{Cht}$.
\end{enumerate}
\end{corollary}

\begin{proof}
Let $K\gg 1$ be the constant that appears in item \eqref{I:en-bd} of Prop \ref{P:main}, and $0 < \omega \ll 1$ be such that $T'=\omega h^{-1}$ in item \eqref{I:time-scale} of Prop. \ref{P:main}.  Let $\kappa_j = \beta K^j$ for $1\leq j \leq J$, where $J$ is such that $K^J \sim h^{-1/4}$.  Let $I_j$ denote the time interval $I_j=[(j-1)\omega h^{-1}, j\omega h^{-1}]$.  Apply Prop. \ref{P:main} on $I_j$ with $\kappa = \kappa_j$.
\end{proof}

Now we complete the proof of Theorem \ref{T:main}.  Recall that we are given $b_0 \in C_c^\infty(\mathbb{R}^2)$, $\delta>0$, $a_0\in \mathbb{R}$, and $c_0>0$.  Let $A(\tau)$, $C(\tau)$, and $T_*$ be given as in Def. \ref{D:time-scale}.  Let $T'$, $a(t)$, $c(t)$ be as given in Cor. \ref{C:main}.  Let $\tilde a(t) = h^{-1}A(ht)$ and $\tilde c(t) = C(ht)$.  Then

$$
\left\{
\begin{aligned}
& \dot {\tilde a} - 4\tilde c^2 + b(\tilde a) = 0 \\
& \dot {\tilde c} - \tfrac13 \tilde c b'(\tilde a) =0
\end{aligned}
\right.
$$
on $0\leq t \leq h^{-1}T_*$.  Then
\begin{align*}
|a - \tilde a|(t) &\leq \int_0^t |\dot a - \dot{\tilde a}| \, ds \\
& \leq \int_0^t | (4c^2-b(a)) - (4\tilde c^2 - b(\tilde a)) |(s) \,ds + \int_0^t |\dot a- 4c^2 + b(a)|(s) \, ds \\
&\lesssim \int_0^t |c-\tilde c|(s) \,ds + h \int_0^t|a-\tilde a| \,ds + \beta^2e^{Cht}
\end{align*}
By Gronwall's inequality,
\begin{equation}
\label{E:ODE1}
| a-\tilde a|(t) \lesssim e^{Cht} \left( \int_0^t |c-\tilde c|(s) \,ds + \beta^2 \right) \,.
\end{equation}
Also,
\begin{align*}
|c -\tilde c|(t) 
&\lesssim \left| \frac{c}{\tilde c} -1 \right|(t) \lesssim \left| \ln \frac{c}{\tilde c} \right|(t) = |\ln c - \ln \tilde c|(t) = \int_0^t \left| \frac{\dot c}{c} - \frac{\dot{\tilde c}}{\tilde c} \right|(s) \, ds \\
&\lesssim \int_0^t |b'(a) - b'(\tilde a) | \, ds + \int_0^t |\frac{\dot c}{c} - \frac13 b'(a)| \,ds \\
&\lesssim h \int_0^t |a-\tilde a| \, ds + \beta^2 he^{Cht}
\end{align*}
Combining, and applying Gronwall's inequality again, we obtain
$$|c-\tilde c|(t) \lesssim \beta^2 he^{Cht} \,.$$
Substitution back into \eqref{E:ODE1} yields
$$|a - \tilde a|(t) \lesssim \beta^2 e^{Cht}$$
This completes the proof of Theorem \ref{T:main}.

\appendix

\section{Global well-posedness}

In this section, we prove that \eqref{E:pKdV} is globally well-posed in $H^1$.  The local well-posedness (Prop. \ref{P:local} below) is a consequence of the local smoothing and maximal function estimate of Kenig-Ponce-Vega \cite{KPV} and the global well-posedness follows from the local well-posedness and the nearly conserved $L^2$ norm and Hamiltonian (Prop. \ref{P:global} below).  A similar argument is given in Apx. A of \cite{DS} with an additional smallness assumption on $b$.  This smallness assumption could be removed by scaling their result.  However, for expository purposes we present a shorter proof here, which also imposes fewer hypotheses on $b$.

In this section, we adopt the notation $L_T^p$ to mean $L_{[0,T]}^p$ and $C_TH_x^s$ to mean $C([0,T]; H_x^s)$, etc.  The ordering of multiple norms is standard: for example, $\|w\|_{L_x^2L_T^\infty} =  \| \, \|w\|_{L_T^\infty} \, \|_{L_x^2}$.

\begin{proposition}[local well-posedness of \eqref{E:pKdV} in $H^1$]
\label{P:local}
Let $X$ be the space of functions on $[0,T]\times \mathbb{R}$ defined by the norm
$$\|w\|_X = \|w\|_{L_x^2L_T^\infty} + \|w\|_{C_{t\in [0,T]} H_x^1}$$
Suppose that 
$$A \defeq \|b\|_{L_x^2L_{t\in [0,1]}^\infty} + \|\partial_x b\|_{L_{t\in [0,1]}^\infty L_x^2}<\infty$$ 
and $\phi\in H^1$.  Then there exists $T=T(A, \|\phi\|_{H^1})\leq 1$ and a  solution $u\in X$ to \eqref{E:pKdV} with initial data $\phi$ on $[0,T]$.  This solution is the unique solution belonging to the function class $X$.  Moreover, the data-to-solution map is Lipschitz continuous.
\end{proposition}

\begin{proof}
Let $U$ denote the linear flow (no potential) operator, a mapping from functions of $x$ to functions of $(x,t)$, defined by
$$(U\phi)(x,t) = e^{-t\partial_x^3}\phi(x) = \frac1{2\pi}\int_\xi e^{ix\xi^3} \hat \phi(\xi) \, d\xi \,.$$ 
Let $I$ denote the Duhamel operator, a mapping from functions of $(x,t)$ to functions of $(x,t)$, defined by
$$(If)(x,t) = \int_0^t e^{-(t-t')\partial_x^3} f(\cdot, t') \, dt' \,.$$
That is, if $w=U\phi$, then $w$ solves the homogeneous initial-value problem $\partial_t w + \partial_x^3w =0$ with $w(0,x)=\phi(x)$.  If $w=If$, then $w$ solves the inhomogeneous initial-value problem $\partial_t w + \partial_x^3w =f$ with $u(0,x)=0$.

Kenig-Ponce-Vega \cite{KPV-JAMS, KPV} establish the estimates
\begin{align}
\label{E:est-1}
&\| U\phi \|_{C_T L_x^2} \leq \|\phi\|_{L_x^2} \\
\label{E:est-2}
&\| U \phi \|_{L_x^2 L_T^\infty} \leq \| \phi \|_{H_x^1} \\
\label{E:est-3}
&\| \partial_x I f \|_{C_T L_x^2} \leq \|f\|_{L_x^1L_T^2} \\
\label{E:est-4}
&\| \partial_x I f \|_{L_x^2L_T^\infty} \leq \| f\|_{L_x^1L_T^2}+ \| \partial_x f\|_{L_x^1L_T^2}
\end{align}
with implicit constants independent of $0\leq T\leq 1$.  In fact, \eqref{E:est-1} is just the unitarity of $U(t)$ on $L_x^2$, \eqref{E:est-2} is (2.12) in Cor. 2.9 in \cite{KPV-JAMS}, \eqref{E:est-3} is (3.7) in Theorem 3.5(ii) in \cite{KPV}, and \eqref{E:est-4} is not explicitly contained in \cite{KPV-JAMS, KPV}, but can be deduced from the above quoted estimates as follows.  By the Christ-Kiselev lemma as stated and proved in Lemma 3 of Molinet-Ribaud \cite{MR}, it suffices to show that
$$\left\| \partial_x \int_0^T U(t-t')f(t') \, dt' \right\|_{L_x^2 L_T^\infty} \lesssim \|  f\|_{L_x^1L_T^2} +  \| \partial_x f\|_{L_x^1L_T^2} \,.$$
By first applying \eqref{E:est-2} and then the dual to the local smoothing estimate $\| \partial_x U \phi \|_{L_x^\infty L_T^2} \lesssim \|\phi\|_{L_x^2}$ (Lemma 2.1 in \cite{KPV-JAMS}), we obtain 
\begin{align*}
\indentalign
\left\| \partial_x \int_0^T U(t-t')f(t') \, dt' \right\|_{L_x^2 L_T^\infty} \\
&\lesssim  \left\|  \partial_x \int_0^T U(-t')f(t') \, dt' \right\|_{L_x^2} +  \left\|  \partial_x \int_0^T U(-t')\partial_x f(t') \, dt' \right\|_{L_x^2} \\
&\lesssim \| f\|_{L_x^1L_T^2} + \|\partial_x f\|_{L_x^1 L_T^2} \,,
\end{align*}
as claimed.

Let $\Phi$ be the mapping 
\begin{equation}
\label{E:cm}
\Phi(w) = U\phi + \partial_x I(w^2-bw) \,,
\end{equation}
We seek a fixed point $\Phi(u)=u$ in some ball in the space $X$.  To control inhimogeneities,  we need the following four estimates, which are consequences of H\"older's inequality:
\begin{align}
\label{E:est-5}
&\| \partial_x (bu) \|_{L_x^1L_T^2} \lesssim T^{1/2}( \|\partial_x b\|_{L_T^\infty L_x^2}\|u\|_{L_x^2L_T^\infty} + \|b\|_{L_x^2L_T^\infty} \|\partial_x u\|_{L_T^\infty L_x^2}) \\
\label{E:est-6}
&\| bu \|_{L_x^1L_T^2} \lesssim T^{1/2} \|b\|_{L_x^2L_T^\infty} \|u\|_{L_T^\infty L_x^2}\\
\label{E:est-7}
&\| \partial_x (u^2) \|_{L_x^1L_T^2} \lesssim T^{1/2}\|u\|_{L_x^2L_T^\infty} \|\partial_x u\|_{L_T^\infty L_x^2} \\
\label{E:est-8}
&\| u^2 \|_{L_x^1L_T^2} \lesssim T^{1/2} \|u\|_{L_x^2L_T^\infty} \|u\|_{L_T^\infty L_x^2}
\end{align}
We prove \eqref{E:est-5}.
\begin{align*}
\| \partial_x (bu) \|_{L_x^1L_T^2} &\leq \| \partial_x b \|_{L_x^2L_T^2} \| u\|_{L_x^2 L_T^\infty} + \|b \|_{L_x^2L_T^\infty} \|\partial_x u\|_{L_x^2L_T^2} \\
&\leq \| \partial_x b \|_{L_T^2L_x^2} \| u\|_{L_x^2 L_T^\infty} + \|b \|_{L_x^2L_T^\infty} \|\partial_x u\|_{L_T^2L_x^2}\\
&\leq T^{1/2}\| \partial_x b \|_{L_T^\infty L_x^2} \| u\|_{L_x^2 L_T^\infty} + T^{1/2}\|b \|_{L_x^2L_T^\infty} \|\partial_x u\|_{L_T^\infty L_x^2}
\end{align*}
which is \eqref{E:est-5}.  The other estimates, \eqref{E:est-6}, \eqref{E:est-7}, \eqref{E:est-8} are proved similarly.

By \eqref{E:est-2}, \eqref{E:est-4},
\begin{equation}
\label{E:est-9}
\| \Phi(w) \|_{L_x^2L_T^\infty} \lesssim \|\phi\|_{H^1} + \| (w^2-bw) \|_{L_x^1L_T^2} + \| \partial_x(w^2-bw)\|_{L_x^1L_T^2}
\end{equation}
By \eqref{E:est-1}, \eqref{E:est-3},
\begin{equation}
\label{E:est-10}
\| \Phi(w) \|_{L_T^\infty L_x^2} \lesssim \| \phi \|_{L_x^2} + \| (w^2-bw) \|_{L_x^1 L_T^2}
\end{equation}
Applying $\partial_x$ to \eqref{E:cm} and estimating with \eqref{E:est-1}, \eqref{E:est-3},
\begin{equation}
\label{E:est-11}
\| \partial_x \Phi(w) \|_{L_T^\infty L_x^2} \lesssim \|\partial_x \phi \|_{L_x^2} + \| \partial_x (w^2-bw) \|_{L_x^1 L_T^2}
\end{equation}
Combining \eqref{E:est-9}, \eqref{E:est-10}, \eqref{E:est-11}, and bounding the right-hand sides using \eqref{E:est-5}, \eqref{E:est-6}, \eqref{E:est-7}, \eqref{E:est-8}, we obtain
\begin{equation}
\label{E:est-12}
\| \Phi(w) \|_X \leq  C\|\phi \|_{H^1} + CT^{1/2}(A\|w\|_{X} + \|w\|_X^2)
\end{equation}
Let $B=2C\|\phi \|_{H^1}$, and consider $X_B = \{ \, w\in X \, | \, \|w\|_X \leq B \}$ and $T\leq \frac1{16} C^{-2} \min(A^{-2}, B^{-2})$.  Then \eqref{E:est-12} implies that $\Phi: X_B \to X_B$.  

We similarly establish that $\Phi$ is a contraction on $X_B$, which completes the proof.
\end{proof}

\begin{proposition}[global well-posedness of \eqref{E:pKdV} in $H^1$]
\label{P:global}
Suppose that $b\in C^1(\mathbb{R}^{1+1})$ and satisfies the following.  
Suppose that for every unit-sized time interval $I$, we have
$$\|b\|_{L_x^2L_{t\in I}^\infty} + \|\partial_x b\|_{L_{t\in I}^\infty L_x^2}<\infty \,.$$
(the bound need not be uniform with respect to all time intervals).  
Also suppose that for all $t$, $$\| \partial_x b(t) \|_{L_x^\infty} < \infty \,, \qquad \|\partial_t b(t) \|_{L_x^\infty} < \infty \,.$$
Let $\phi\in H^1$.  Then the local $H^1$ solution to \eqref{E:pKdV} with initial data $\phi$ given by Prop. \ref{P:local} extends to a global solution with
$$\|u(t) \|_{H^1} \lesssim  \la \|\phi \|_{H^1} \ra^4  \left(\| b\|_{L_{[0,t]}^\infty L_x^\infty} + \int_0^t  \|b_t(s)\|_{L_x^\infty} e^{\gamma(s)} \, ds\right) \,,$$
where $\gamma(s)$ is given by
$$\gamma(t) = \int_0^t \|b_x(s)\|_{L_{[0,s]}^\infty L_x^\infty} \, ds \,.$$
\end{proposition}
\begin{proof}
Let $P(t) = \|u(t)\|_{L^2}^2$ (the momentum) and recall the definition \eqref{E:Hamiltonian} of $H$, the Hamiltonian.  Direct computation shows that
$$\partial_t P = \int  b_x u^2 \, dx \,, \qquad \partial_t H = \frac12 \int b_t u^2 \, dx \,.$$
 Then $|P'(t)| \leq \gamma'(t)P(t)$, and hence $\partial_t [e^{-\gamma(t)}P(t)] \leq 0$.  From this, we conclude that 
$$P(t) \leq e^{\gamma(t)}P(0) \,.$$  
In addition, we have
$$|H'(t)| \leq \|b_t(t)\|_{L_x^\infty} P(t) \leq \|b_t(t)\|_{L_x^\infty} e^{\gamma(t)} P(0)$$
Hence
$$H(t) \leq H(0) + P(0) \int_0^t  \|b_t(s)\|_{L_x^\infty} e^{\gamma(s)} \, ds$$
By the Gagliardo-Nirenberg inequality $\|u\|_{L^3}^3 \leq \|u\|_{L^2}^{5/2}\|\partial_x u\|_{L^2}^{1/2}$ and the Peter-Paul inequality, we have $\|u\|_{L^3}^3 \leq \frac18 \|u_x\|_{L_x^2}^2 + C \|u\|_{L_x^2}^{10/3}$.  Hence
$$\|u_x\|_{L_x^2}^2 \leq C \|u\|_{L_x^2}^{10/3} + \|b(t)\|_{L_x^\infty}\|u(t)\|_{L_x^2}^2 + H(t)$$
When combined with the inequalities for $H(t)$ and $P(t)$, this gives the conclusion. 
\end{proof}


\begin{thebibliography}{00}


\bibitem{Benjamin} T. Benjamin, \emph{The stability of solitary waves}, Proc. Roy. Soc. (London) Ser. A 328 (1972) pp. 153--183. 

\bibitem{Bona} J. Bona, \emph{On the stability theory of solitary waves}, Proc. Roy. Soc. London Ser. A 344 (1975) pp. 363--374.

\bibitem{BSS} J.L. Bona, P.E. Souganidis, and W.A. Strauss, 
\emph{Stability and instability of solitary waves of Korteweg de Vries type},
Proc. Roy. Soc. London Ser. A 411 (1987) pp. 395--412. 




\bibitem{DS} S.I. Dejak and I.M. Sigal, 
\emph{Long time dynamics of KdV solitary waves over a variable bottom}, 
Comm. Pure Appl. Math. 59 (2006) pp. 869--905.

\bibitem{FGJS} J. Fr\"ohlich, S. Gustafson, B.L.G. Jonsson, and
I.M. Sigal, {\em Solitary wave dynamics in an external potential,}
Comm. Math. Physics 250 (2004) pp. 613--642.

\bibitem{GZ} L. Guillop\'e and M. Zworski, 
\emph{Upper bounds on the number of resonances on 
noncompact Riemann surfaces}, 
J. Func. Anal. 129 (1995) pp. 364-389.

\bibitem{HPZ} J. Holmer, G. Perelman, and M. Zworski, \emph{Effective dynamics of double solitons for perturbed mKdV}, to appear in Comm. Math. Phys., arxiv.org preprint \texttt{arXiv:0912.5122 [math.AP]}.  The numerical illustrations of the results and MATLAB codes can be found at \texttt{http://math.berkeley.edu/~zworski/hpzweb.html}.

\bibitem{HZ1} J. Holmer and M. Zworski, \emph{Slow soliton interaction with delta impurities}, J. Modern Dynamics 1 (2007) pp. 689--718.

\bibitem{HZ2} J. Holmer and M. Zworski, \emph{Soliton interaction with slowly varying potentials}, IMRN Internat. Math. Res. Notices 2008 (2008), Art. ID runn026, 36 pp. 

\bibitem{HZ3} J. Holmer and M. Zworski, \emph{Geometric structure of NLS evolution}, unpublished note available at \texttt{http://math.brown.edu/$\sim$holmer}.

\bibitem{KT} A.-K. Kassam and L.N. Trefethen, 
\emph{Fourth-order time-stepping for stiff PDEs},
SIAM J. Sci. Comput. 26 (2005) pp. 1214--1233. 

\bibitem{KPV-JAMS} C.E. Kenig, G. Ponce, and L. Vega, \emph{Well-posedness of the initial value problem for the Korteweg-de Vries equation},  J. Amer. Math. Soc.  4  (1991) pp. 323--347. 

\bibitem{KPV} C.E. Kenig, G. Ponce, L. Vega, \emph{Well-posedness and scattering results for the generalized Korteweg-de Vries equation via the contraction principle},  Comm. Pure Appl. Math.  46  (1993) pp. 527--620.

\bibitem{MM2} Y. Martel and F. Merle, \emph{Asymptotic stability of solitons for subcritical generalized KdV equations}, Arch. Ration. Mech. Anal. 157 (2001) pp. 219--254.

\bibitem{MM} Y. Martel and F. Merle, \emph{Asymptotic stability of solitons for subcritical gKdV equations revisited}, Nonlinearity 18 (2005) pp. 55--80.

\bibitem{MR}  L. Molinet and F.  Ribaud, \emph{Well-posedness results for the generalized Benjamin-Ono equation with small initial data},  J. Math. Pures Appl. (9)  83  (2004) pp. 277--311.

\bibitem{Munoz} C. Munoz, \emph{On the soliton dynamics under a slowly varying medium for generalized KdV equations},  arxiv.org \texttt{arXiv:0912.4725 [math.AP]}.

\bibitem{Wei} M.I. Weinstein, {\em Lyapunov stability of ground
states of nonlinear dispersive evolution equations,}
Comm.~Pure.~Appl.~Math. 29 (1986) pp. 51--68.

\end{thebibliography}
\end{document}